\documentclass[12pt]{amsart}
\usepackage{amsmath,amssymb,amscd,amsthm,amsfonts}
\usepackage[all,cmtip]{xy} 
\usepackage{hyperref}
\usepackage [numbers] {natbib}
\usepackage{enumerate}

\newtheorem{thm}{Theorem}[section]

\newtheorem{prop}[thm]{Proposition}
\newtheorem{cor}[thm]{Corollary} 
\theoremstyle{definition}
\newtheorem{defn}[thm]{Definition} 
\theoremstyle{remark}

\newtheorem{rmk}[thm]{Remark}
\numberwithin{equation}{section}

\newcommand{\fa}{\mathfrak{a}}

\renewcommand{\AA}{\mathbb{A}}
\newcommand{\CC}{\mathbb{C}}
\newcommand{\FF}{\mathbb{F}}

\newcommand{\RR}{\mathbb{R}} 
\newcommand{\ZZ}{\mathbb{Z}}

\newcommand{\A}{\mathbb{A}}

\newcommand{\cA}{\mathcal{A}}

\newcommand{\cD}{\mathcal{D}}
\newcommand{\cE}{\mathcal{E}}

\newcommand{\cH}{\mathcal{H}}

\newcommand{\cP}{\mathcal{P}}
\newcommand{\cS}{\mathcal{S}}

\newcommand{\cusp}{\mathrm{cusp}}

\newcommand{\half}{\frac {1} {2}}

\newcommand{\til}{\widetilde}

\DeclareMathOperator{\BC}{\mathrm{BC}}
\DeclareMathOperator{\Rat}{\mathrm{Rat}}
\DeclareMathOperator{\Stab}{\mathrm{Stab}}

\DeclareMathOperator{\GL}{\mathrm{GL}}
\DeclareMathOperator{\SL}{\mathrm{SL}}

\DeclareMathOperator{\Sym}{\mathrm{Sym}}

\DeclareMathOperator{\Sp}{\mathrm{Sp}}
\DeclareMathOperator{\Mp}{\mathrm{Mp}}

\DeclareMathOperator{\Ind}{\mathrm{Ind}}

\DeclareMathOperator{\Hom}{\mathrm{Hom}}

\DeclareMathOperator{\Res}{\mathrm{Res}} 
\DeclareMathOperator{\res}{\mathrm{res}} 
\newcommand{\isom}{\cong}

\DeclareMathOperator{\disc}{\mathrm{disc}}

\DeclareMathOperator{\vol}{\mathrm{vol}}

\DeclareMathOperator{\lmod}{\backslash}

\DeclareMathOperator{\LO}{\mathrm{LO}}
\DeclareMathOperator{\FO}{\mathrm{FO}} 
\let\Re\undefined
\DeclareMathOperator{\Re}{\mathrm{Re}}

\newcommand{\form}[2]{\langle{#1},{#2}\rangle}
\newcommand  {\openclose} [2] {(#1,#2]} 
\usepackage[left=1.5in,right=1.5in, top=1.5in, bottom=1.5in]{geometry}
\begin{document} \title[Periods and $(\chi,b)$-factors of Cuspidal
Forms of $\Mp(2n)$] {Periods and $(\chi,b)$-factors of Cuspidal
Automorphic Forms of Metaplectic Groups}

\author{Chenyan Wu} \address{School of Mathematics and Statistics, 
The University of Melbourne,
Victoria, 3010, Australia}
\email{chenyan.wu@unimelb.edu.au}


\thanks{The research is supported in part by General Program of National Natural Science Foundation of China (11771086),
  National Natural Science Foundation of China (\#11601087) and by Program of
  Shanghai Academic/Technology Research Leader (\#16XD1400400).}

\date{\today}
\begin{abstract}
We give constraints on the existence of $(\chi,b)$-factors in the global $A$-parameter of a genuine cuspidal automorphic representation $\sigma$ of a metaplectic group in terms of  the invariant, lowest occurrence index, of theta lifts to odd orthogonal groups. We also give a refined result that relates the invariant, first occurrence index, to non-vanishing of period integrals of residues of Eisenstein series associated to the cuspidal datum $\chi\otimes\sigma$. This complements our previous results for symplectic groups.
\end{abstract}
\maketitle{}

\section*{Introduction}
\label{sec:introduction}
Let $G$ be a classical group over a number field $F$. Let $\AA$ be the ring of adeles of $F$. Let $\sigma$ be an
automorphic representation of $G (\AA)$ in the discrete spectrum, though later we consider only cuspidal automorphic
representations. By the theory of endoscopic classification developed by
\cite{MR3135650} and extended by \cite {MR3338302,kaletha:_endos_class_repres}, one attaches a global $A$-parameter to $\sigma$. These works
depend on the stabilisation of the twisted trace formula which has been established by a series of works by M\oe glin
and Waldspurger. We refer the readers to their books \cite{MR3823813,MR3823814}. Via the Shimura--Waldspurger
correspondence of the metaplectic group $\Mp_{2n} (\AA)$, which is the non-trivial double cover of the symplectic group
$\Sp_{2n} (\AA)$, and odd special orthogonal groups, Gan and Ichino \cite{MR3866889} also attached global $A$-parameters to
genuine automorphic representations of $\Mp_{2n} (\AA)$ in the discrete spectrum. This depends on the choice of an
additive character $\psi$ that is used in the Shimura--Waldspurger correspondence. In this introduction, we also let $G$ denote $\Mp_{2n}$ by
abuse of notation and we suppress the dependence on $\psi$. The $A$-parameter $\phi (\sigma)$ of
$\sigma$ is a formal sum
\begin{equation*}
  \phi (\sigma) = \boxplus_i (\tau_i, b_i)
\end{equation*}
where $\tau_i$ is an irreducible self-dual cuspidal automorphic representation of some $\GL_{n_i} (\AA)$,
$b_i$ is a positive integer which represents the unique $b_i$-dimensional irreducible representation of Arthur's
$\SL_2 (\CC)$, together with some conditions on $(\tau_i,b_i)$ so that the type is compatible with the type of the dual 
group of $G$.  When $G$ is a unitary group, then we need to introduce a quadratic field extension of $F$. As this case is not the focus of this article, we refer the readers to \cite[Sec.~1.3]{wu21:chi-b-unitary_ii} for details. The exact conditions for $\Mp_{2n}$ are spelled out in
Sec.~\ref{sec:arthur-parameters}. Here we have adopted the notation of \cite{jiang14:_autom_integ_trans_class_group_i}
where the theory of $(\tau,b)$ was introduced.

One principle of the $(\tau,b)$-theory is that if $\phi (\sigma)$ has $(\tau,b)$ as a simple factor, then there should
exist a kernel function constructed out of $(\tau,b)$ that transfers $\sigma$ to a certain automorphic representation
$\cD (\sigma)$ in the $A$-packet attached to `$\phi(\sigma) \boxminus (\tau,b)$', i.e., $\phi(\sigma)$ with the factor $(\tau,b)$ removed.
From the work of Rallis \cite{MR658543} and Kudla--Rallis \cite{MR1289491}, it is clear when $\tau$ is a character $\chi$, the kernel function is the theta kernel possibly twisted by $\chi$. (See Sec.~7 of \cite{jiang14:_autom_integ_trans_class_group_i}.)

Now we focus on the case of the metaplectic groups. We will put the additive character $\psi$ back into the
notation to emphasise, for example, that the $L$-functions etc. depend on it. Let $X$ be a
$2n$-dimensional symplectic space.  We will now write $\til
{G} (X)$ for $\Mp_{2n}$ to match the notation in the body of the article.

The poles of the tensor product $L$-function
$L_\psi (s,\sigma\times\tau)$ detect the existence of $(\tau,b)$-factors in the $A$-parameter $\phi_\psi (\sigma)$. As a first step, we consider the case when $\tau$ is a quadratic automorphic character $\chi$ of $\GL_1 (\AA)$. The $L$-function $L_\psi (s,\sigma\times\chi)$ has been well-studied.
 By the regularised Rallis inner product formula \cite{MR1289491} which was proved using the regularised
Siegel-Weil formula also proved there and the doubling method construction  \cite{ps-rallis87_l_funct_for_the_class_group} (see also \cite{MR2192828,MR3006697,MR3211043}), the partial $L$-function $L_\psi^S (s,\sigma\times\chi)$ detects whether the theta
lifts of $\sigma$ to the odd orthogonal groups in certain Witt towers vanish or not. The complete theory that interprets the existence
of poles or non-vanishing of the complete
$L$-function as the obstruction to the local-global principle of  theta correspondence was done in \cite{MR3211043}
and then extended to the `second term' range by \cite{MR3279536}. We will relate certain invariants of the theta lifts of $\sigma$ to the existence of $(\chi,b)$-factors in the $A$-parameter $\phi_\psi(\sigma)$.

We now describe our main results which are concerned with the metaplectic case, while pointing out the similarity to the statements in the symplectic case. Let $\LO_{\psi,\chi} (\sigma)$ denote the lowest occurrence index of $\sigma$ in the theta
correspondence to all Witt towers associated to the quadratic character $\chi$. (See
Sec.~\ref{sec:first-occurr-lowest} for the precise definition.)  Let $\cA_\cusp (\til {G} (X))$ denote the
set of all irreducible genuine cuspidal automorphic representations of $\til {G} (X) (\AA)$. We note
that there is no algebraic group $\til {G} (X)$ and that the notation is used purely for aesthetic
reason. Then we get the following constraint on $(\chi,b)$-factors of $\phi_\psi(\sigma)$ in terms of the
invariant $\LO_{\psi,\chi} (\sigma)$.

\begin{thm}[Thm.~\ref{thm:chi-b-factor-LO}]
   Let $\sigma\in\cA_\cusp (\til {G} (X))$ and $\chi$ be a quadratic automorphic character of $\AA^\times$. Let $\phi_\psi (\sigma)$ denote
   the A-parameter of $\sigma$ with respect to $\psi$. Then the following hold.
   \begin{enumerate}
   \item If $\phi_\psi (\sigma)$ has a $(\chi,b)$-factor, then $b\le \half\dim X +1$.
  \item If $\phi_\psi (\sigma)$ has a $(\chi,b)$-factor with $b$ maximal among all simple factors of $\phi_\psi (\sigma)$, then  $\LO_{\psi,\chi}    (\sigma) \le \dim X - b +1$.
  \item If $\LO_{\psi,\chi} (\sigma) = 2j +1 < \dim X+2$, then $\phi_\psi (\sigma)$ cannot have a $(\chi, b)$-factor with $b$ maximal among all simple factors of $\phi_\psi (\sigma)$ and $b>\dim
    X -2j$.
  \item If $\LO_{\psi,\chi} (\sigma) = 2j +1 > \dim X+2$, then  $\phi_\psi (\sigma)$ cannot have a $(\chi, b)$-factor with $b$ maximal among all simple factors of $\phi_\psi (\sigma)$.
  \end{enumerate}
\end{thm}

 \begin{rmk}
   Coupled with the results of \cite{MR3805648}, we see that in both the $\Sp$ and $\Mp$ cases, if $\phi (\sigma)$ has a $(\chi,b)$-factor with $b$ maximal among all simple factors, then  $\LO_{\chi,\psi}    (\sigma) \le \dim X - b +1$.
 \end{rmk}

 The symplectic analogue of our result on the  existence of $(\chi,b)$-factors in the $A$-parameter of $\sigma$ is a key input to \cite[Sec.~5]{MR3969876} which  gives a bound on the exponent that measures the departure from  temperedness of the local components of cuspidal automorphic representations.  In this sense,  our results have bearings on the generalised Ramanujan conjecture as proposed in \cite{MR2192019}.

 Following an idea of M\oe glin's in \cite{MR1473165} which considered the even orthogonal  and the
symplectic case, we consider, instead of  the partial $L$-function $L^S (s,\sigma\times\chi)$, the Eisenstein series $E (g,f_s)$ attached to the
cuspidal datum $\chi\otimes\sigma$ (c.f. Sec.~\ref{sec:eisenstein-series}). The theorem above is derived from our results on poles of $E (g,f_s)$. We get much more precise relation between poles of the
Eisenstein series and $\LO_{\psi,\chi} (\sigma)$. This line of investigation has been taken up by \cite{MR2540878}
in the odd orthogonal group case and by \cite{MR3435720,wu21:chi-b-unitary_ii} in the unitary group case. We treat the
metaplectic group case in this paper. Let $\cP_1 (\chi,\sigma)$ denote the set of all positive poles
of the Eisenstein series we consider. Then we get:

\begin{thm} [Thm.~\ref{thm:strengthend-eis-pole-2-LO}]
    Let $\sigma\in\cA_\cusp (\til {G} (X))$ and let $s_0$ be the maximal element in $\cP_1 (\chi,\sigma)$. Write
    $s_0=\half (\dim X +1)-j$. Then
  \begin{enumerate}
  \item $j$ is an integer such that $\frac{1}{4} (\dim X -2) \le j < \half (\dim X +1)$.
  \item $\LO_{\psi,\chi} (\sigma) = 2j+1$.
  \end{enumerate}
\end{thm}
\begin{rmk}
  Coupled with the results of \cite{MR3805648}, we see that $s_0$ is a non-integral half-integer in
  $\openclose{0} {\frac{1}{4}\dim X +1}$ in the $\Mp$ case and an integer in the same range in the
  $\Sp$ case; and $\LO_{\psi,\chi} (\sigma) = \dim X +2 -2s_0$ in both cases.
\end{rmk}

It is known by Langlands' theory \cite{MR0419366} of Eisenstein series that the pole in the theorem is simple. (See also
\cite{MR1361168}.)  We note that the equality in the theorem is actually informed by considering
periods of residues of the Eisenstein series. We get only the upper bound $2j+1$ by considering only
poles of Eisenstein series. The odd orthogonal case was considered in \cite{MR2540878}, the
symplectic case in \cite{MR3805648} and  the unitary case in \cite{wu21:chi-b-unitary_ii}. Each case has its own technicalities. In this paper, we extend all results
in the symplectic case to the metaplectic case to complete the other half of the picture. As pointed
out in \cite {MR3006697}, unlike the orthogonal or symplectic case, the local factors change in the
metaplectic case when taking contragredient. Thus it is important, in the present case, to employ
complex conjugation in the definition of the global theta lift.

Let $\FO_{\psi}^Y (\sigma)$ denote the first occurrence index of $\sigma$ in the Witt tower
of $Y$ where $Y$ is a quadratic space (c.f. Sec.~\ref{sec:first-occurr-lowest}). We get a result on
$\FO_{\psi}^Y (\sigma)$ and the non-vanishing of periods of residues of the Eisenstein series (or distinction by subgroups
of $G (X)$). Note that
$\FO_{\psi}^Y (\sigma)$ is a finer invariant than $\LO_{\psi,\chi} (\sigma) $ while the non-vanishing
of periods of
residues provides more information  than   locations of  simple poles do. The following theorem is part of
Cor.~\ref{cor:FO-2-distinction-res-eis-theta}. Let $\cE_{s_0} (g,f_s)$ denote the residue
of the Eisenstein series $E (g,f_s)$ at $s=s_0$ and let  $\theta_{\psi,X_1,Y} (g,1,\Phi)$ be the theta function attached to the
Schwartz function $\Phi$. The space $X_1$ is the symplectic space formed by adjoining a
hyperbolic plane to $X$ and similarly $Z_1$ is formed by adjoining the same hyperbolic place to $Z$. Thus $Z_1\subset X_1$. Let $L\subset Z \subset Z_1$ be a totally isotropic subspace. Then the  group $J (Z_1,L)$ is defined to be the
subgroup of $G (Z_1)$ that fixes  $L$ element-wise.
\begin{thm} 
  Let $\sigma\in\cA_\cusp (\til {G} (X))$ and let $Y$ be an anisotropic quadratic space of odd dimension. Assume that
  $\FO_{\psi}^Y (\sigma)= \dim Y+2r$. Let $s_0 = \half (\dim X - (\dim Y +2r ) +2)$. Then $s_0$ is a pole of $E (g,f_s)$ for  some choice of the section $f_s$. Assume further that $s_0>0$. Then there exist a non-degenerate symplectic subspace $Z$ of $X$ and an isotropic subspace $L$ of $Z$
satisfying $\dim X - \dim Z + \dim L = r$ and $\dim L \equiv r \pmod {2}$ with $\dim L=0$ or $1$ such that
  $\cE_{s_0} (g,f_s)\overline {\theta_{\psi,X_1,Y} (g,1,\Phi)}$ is $J (Z_1,L)$-distinguished for  some choice of the section $f_s$ and the Schwartz function $\Phi$.
\end{thm}

We hope our period integrals can inform problems in Arithmetic Geometry. For example, similar work
in \cite{MR2540878} has been used in \cite{MR3508219}. (See also \cite {MR3821921} and
\cite{MR3012154}). It should be pointed out that unlike the orthogonal or the unitary case, our
period integrals may involve a Jacobi group $J (Z_1,L)$ when $L$ is non-trivial (which occurs half
of the time). This is to account for the lack of `odd symplectic group'.

The structure of the paper is as follows. Since we aim to generalise results of symplectic groups to metaplectic groups, we try not deviate too much from the structure of \cite{MR3805648}. Many of the results there generalise readily, in which case, we put in a few words to explain why this is so. However at several points, new input is needed to get the proof through, in which case, we derive the results in great detail, noting all the subtleties. We hope in this way, this article makes the flow of arguments more evident than   \cite{MR3805648} in which we dealt with many technicalities.

In Sec.~\ref{sec:notat-prel}, we set up some common notation and in
Sec.~\ref{sec:arthur-parameters}, we describe the global $A$-parameters of genuine automorphic representations of the
metaplectic group following \cite{MR3866889}. In Sec.~\ref{sec:eisenstein-series}, we introduce the Eisenstein series
attached to the cuspidal datum $\chi\otimes\sigma$ and deduce results on their maximal poles. In
Sec.~\ref{sec:first-occurr-lowest}, we define the two invariants, the first occurrence index and the lowest occurrence index of
theta lifts. We deduce a preliminary result on the relation between the lowest occurrence index and the maximal pole of the Eisenstein series and  hence get a characterisation of the locations of poles of the partial $L$-function.  In this way, we are able to show that the lowest occurrence index poses constraints on the existence of
$(\chi,b)$-factors in the $A$-parameter of $\sigma$. Then in Sec.~\ref{sec:four-coeff-theta}, we turn to studying Fourier
coefficients of theta lifts and we get results on vanishing and non-vanishing of certain periods on $\sigma$. Inspired
by these period integrals, in Sec.~\ref{sec:peri-eisenst-seri}, we study the augmented integrals which are periods of
residues of Eisenstein series. We make use of the Arthur truncation to regularise the integrals. We derive results on first
occurrence indices and the vanishing and the non-vanishing of periods of residues of Eisenstein series. To evaluate our period integral, we cut it into many pieces according to certain orbits in a flag variety in Sec.~\ref{sec:peri-eisenst-seri} and  we
devote Sec.~\ref{sec:computation} to the computation of each piece.

\section*{Acknowledgement}
\label{sec:acknowledgement}
The author would like to thank her post-doctoral mentor, Dihua Jiang, for introducing her to the subject and leading her
down this fun field of research.

\section{Notation and Preliminaries}
\label{sec:notat-prel}

Since this article aims at generalising results in \cite{MR3805648} to the case of metaplectic groups, we will adopt
similar notation to what was used there.

Let $F$ be a number field and $\AA  =\AA_F$ be its ring of adeles. Let $(X, \form{\ }{\ }_X)$ be a non-degenerate
symplectic space over $F$. For a non-negative integer $a$, let $\cH_a$ be the direct sum of $a$ copies of the hyperbolic
plane. We form the augmented symplectic space $X_a$ by adjoining $\cH_a$ to $X$. Let $e_1^+,\ldots, e_a^+,e_1^-,\ldots ,
e_a^-$ be a basis of $\cH_a$ such that $\form{e_i^+}{e_j^-}=\delta_{ij}$ and
$\form{e_i^+}{e_j^+}=\form{e_i^-}{e_j^-}=0$, for $i,j=1,\ldots , a$, where $\delta_{ij}$ is the Kronecker delta. Let
$\ell_a^+$ (resp. $\ell_a^-$) be the span of $e_i^+$'s (resp. $e_i^-$'s). Then $\cH_a = \ell_a^+ \oplus \ell_a^-$ is a
polarisation of $\cH_a$. Let $G (X_a)$ be the isometry group of $X_a$ with the action on the right.

Let $L$ be an isotropic subspace of $X$. Let $Q (X,L)$ denote the  parabolic subgroup of $G (X)$ that stabilises $L$ and
let $N (X,L)$ denote its unipotent radical. We write $Q_a$ (resp. $N_a$) for $Q (X_a,\ell_a^-)$ (resp. $N
(X_a,\ell_a^-)$) as shorthand. We also write $M_a$ for the Levi subgroup of $Q_a$ such that with respect to the `basis'
$\ell_a^+,X,\ell_a^-$, $M_a$ consists of elements of the form
\begin{equation*}
  m (x,h)=
  \begin{pmatrix}
    x & & \\ & h & \\ && x^*
  \end{pmatrix} \in G (X_a)
\end{equation*}
for $x\in\GL_a$ and $h\in G (X)$ where $x^*$ is the adjoint of $x$. Fix a  maximal compact subgroup $K_{G (X),v}$ of $G
(X) (F_v)$ for each place $v$ of $F$ and set $K_{G (X)} = \prod_v K_{G (X),v}$. We require that $K_{G (X)}$ is good in the sense that the Iwasawa
decomposition holds and that it is compatible with the Levi decomposition (c.f. \cite [I.1.4] {MR1361168}). Most often we
abbreviate $K_{G (X_a)}$ as $K_a$.

Next we describe the metaplectic groups. Let $v$ be a place of $F$. Let $\til {G} (X_a) (F_v)$ denote the
metaplectic double cover of $G (X_a) (F_v)$. It sits in the unique non-trivial central extension
\begin{equation*}
 1\rightarrow  \mu_2 \rightarrow \til {G} (X_a) (F_v) \rightarrow G (X_a) (F_v) \rightarrow 1
\end{equation*}
unless $F_v =\CC$, in which case, the double cover splits. Let $\til {G} (X_a) (\AA)$ be the double cover of
$G (X_a) (\AA)$ defined as the restricted product of $\til {G} (X_a) (F_v)$ over all places $v$ modulo the group
\begin{equation*}
 \{ (\zeta_v) \in \oplus_v \mu_2 | \prod_v \zeta_v = 1\}.
\end{equation*}
We note that there is a canonical lift of $G (X) (F)$ to $\til {G} (X) (\AA)$ and that $\til {G} (X) (F_v)$ splits
uniquely over any unipotent subgroup of $G (X) (F_v)$.

Fix a non-trivial additive character $\psi: F\lmod \AA \rightarrow \CC^1$. It will be used in the construction of the  Weil
representation that underlies the theta correspondence and the definition of Arthur parameters for cuspidal automorphic
representations of the metaplectic group. For any subgroup $H$ of $G (X_a) (\AA)$ (resp. $G (X_a) (F_v)$)
we write $\til {H}$ for its preimage in $\til {G} (X_a) (\AA)$ (resp. $\til {G} (X_a) (F_v)$). We note that
$\GL_a (F_v)$ occurs as a direct factor of $M_a (F_v)$ and on $\til {\GL}_a (F_v)$ multiplication is given by
\begin{equation*}
  (g_1,\zeta_1) (g_2, \zeta_2) = (g_1g_2,\zeta_1\zeta_2 (\det (g_1), \det (g_2))_v)
\end{equation*}
which has a Hilbert symbol twist when multiplying the $\mu_2$-part. Analogous to the notation $m (x,h)$ above, we write
\begin{equation*}
  \til {m} (x,h,\zeta) =
  (\begin{pmatrix}
   x & & \\ & h & \\ && x^*
   \end{pmatrix},\zeta) \in \til {G} (X_a) (\AA)
\end{equation*}
for $x\in\GL_a (\AA)$, $h\in G (X) (\AA)$ and $\zeta\in\mu_2$. There is also a local version.

As we integrate over automorphic quotients often, we write $[H]$ for $H (F)\lmod H (\AA)$ if $H$ is an algebraic group
and $[\til {H}]$ for $H (F)\lmod \til {H} (\AA)$ if $\til {H} (\AA)$ covers a subgroup $H (\AA)$ of a metaplectic group.

Let $\cA_\cusp (\til {G} (X))$ denote the set of all irreducible genuine cuspidal automorphic representations of $\til
{G} (X) (\AA)$. We note that there is no algebraic group $\til {G} (X)$ and that it is used purely for aesthetic reason.

\section{Arthur Parameters}
\label{sec:arthur-parameters}

We recall the description of $A$-parameters attached to genuine irreducible cuspidal automorphic representations of
metaplectic groups from \cite{MR3866889}. The $A$-parameters are defined via those attached to irreducible automorphic
representations of odd special orthogonal groups via Shimura--Waldspurger correspondence. In this section, we write $\Mp_{2n}$ rather than $\til {G} (X)$.

The global elliptic $A$-parameters $\phi$ for $\Mp_{2n}$  are of the form:
\begin{equation*}
  \phi = \boxplus_i (\tau_i,b_i)
\end{equation*}
where
\begin{itemize}
\item $\tau_i$ is an irreducible self-dual cuspidal automorphic representation of $\GL_{n_i} (\AA)$;
\item $b_i$ is a positive integer which represents the unique $b_i$-dimensional irreducible representation of $\SL_2 (\CC)$
\end{itemize}
such that
\begin{itemize}
\item $\sum_{i} n_ib_i = 2n$;
\item if $b_i$ is odd, then $\tau_i$ is symplectic, that is, $L (s,\tau_i,\wedge^2)$ has a pole at $s=1$;
\item if $b_i$ is even, then $\tau_i$ is orthogonal, that is, $L (s,\tau_i,\Sym^2)$ has a pole at $s=1$;
\item the factors $(\tau_i,b_i)$ are pairwise distinct.
\end{itemize}

Given a global elliptic  $A$-parameter $\phi$ for $\Mp_{2n}$, for each place $v$ of $F$, we can attach a local $A$-parameter
\begin{equation*}
  \phi_v = \boxplus_i (\phi_{i,v},b_i)
\end{equation*}
where $\phi_{i,v}$ is the $L$-parameter of $\tau_{i,v}$ given by the local Langlands
correspondence \cite{MR1011897,MR1876802,MR1738446,MR3049932} for $\GL_{n_i}$. We can write the local $A$-parameter
 $\phi_v$ as a homomorphism
\begin{equation*}
  \phi_v: L_{F_v} \times \SL_2 (\CC) \rightarrow \Sp_{2n} (\CC)
\end{equation*}
where $L_{F_v}$ is the local Langlands group, which is the Weil--Deligne group of $F_v$ for $v$ non-archimedean and is
the Weil group of $F_v$ for $v$ archimedean.
 We associate to it the $L$-parameter
\begin{align*}
  \varphi_{\phi_v}: L_{F_v}  \rightarrow \Sp_{2n} (\CC)\\
  w \mapsto \phi_v (w,
  \begin{pmatrix}
    |w|^{\half} & \\ & |w|^{-\half}
  \end{pmatrix}
).
\end{align*}

Let $L^2 (\Mp_{2n})$ denote the subspace of $L^2 (\Sp_{2n} (F) \lmod \Mp_{2n} (\AA))$ on which $\mu_2$ acts as the
non-trivial character. Let $L^2_{\disc} (\Mp_{2n})$ denote the subspace of $L^2 (\Mp_{2n})$ that is the direct sum of
all irreducible sub-representations of $L^2 (\Mp_{2n})$ under the action of $\Mp_{2n} (\AA)$. Let
$L^2_{\phi,\psi} (\Mp_{2n})$ denote subspace of $L_{\disc}^2 (\Mp_{2n})$ generated by the full near equivalence class of
$\pi \subset L_{\disc}^2 (\Mp_{2n})$ such that the $L$-parameter of $\pi_v$ with respect to $\psi$ is $\varphi_{\phi_v}$
for almost all $v$. We note here that for a genuine irreducible automorphic representation of the metaplectic group, its
$L$-parameter depends on the choice of an additive character $\psi$. If one changes $\psi$ to $\psi_a := \psi (a\cdot)$, the $L$-parameter
changes in the way prescribed in \cite [Remark~1.3] {MR3866889}. See also \cite{gan:_repres_metap_group_i}. Then we have
the decomposition:

\begin{thm} [{\cite[Theorem~1.1]{MR3866889}}]
  \begin{equation*}
  L^2_{\disc} (\Mp_{2n}) = \oplus_{\phi} L^2_{\phi,\psi} (\Mp_{2n}).
\end{equation*}
where $\phi$ runs over global elliptic $A$-parameters for $\Mp_{2n}$.
\end{thm}

Thus we see that the $A$-parameter of $\sigma\in\cA_\cusp (\Mp_{2n})$ has the factor  $(\chi,b)$ where $\chi$
is a quadratic automorphic character of $\GL_1 (\AA)$ and $b$ maximal among all factors if
and only if the partial $L$-function $L^S_\psi (s,\sigma\times\chi)$ has its rightmost pole at $s= (b+1)/2$. The
additive character $\psi$ used in the definition of the $L$-function is required to agree with the one used in defining
the $A$-parameter of $\sigma$.

\begin{rmk}
  We note that the partial $L$-function $L^S_\psi (s,\sigma\times\chi)$ is as defined in \cite{MR3006697} and \cite[Sec.~6]{MR3006697} shows that at each unramified place $v$, the local $L$-factor is equal to $L_v(s,\BC_\psi(\sigma_v)\otimes\chi_v)$ where $\BC_\psi$ denotes the base change of $\sigma_v$ to $\GL_{2n}(F_v)$ with respect to $\psi_v$. The $L$-function is $L^S_{\psi^{-1}}(s,\sigma^\vee\otimes\chi)$ in the notation of \cite{MR3211043}. Note that the degenerate principal series $I_\psi(s,\chi)$ defined in \cite[Sec.~3.1]{MR3211043} is $I_{\bar{P}(W^\Delta),\psi^{-1}}(s,\chi)$ in the notation of \cite[Sec.~3]{MR3006697} due to the use of right action in \cite{MR3211043} and this accounts for the of use of $\psi^{-1}$, but then by \cite[Prop.~5.4]{MR3211043}, it is equal to $L^S_\psi(s,\sigma\times\chi)$.
\end{rmk}

\section{Eisenstein series}
\label{sec:eisenstein-series}

For $\sigma\in\cA_\cusp (\Mp_{2n})$ and a quadratic automorphic character $\chi$ of $\AA^\times$, the $L$-function
$L_\psi (s,\sigma\times\chi)$ appears in the constant term of the Eisenstein series on $\til {G} (X_a) (\AA)$ attached to the
maximal parabolic subgroup $Q_a$. We study the poles of these Eisenstein series
in this section.

Let $\fa_{M_a}^* = \Rat (M_a)\otimes_\ZZ \RR$ and $\fa_{M_a,\CC}^*= \Rat (M_a)\otimes_\ZZ \CC $ where $\Rat (M_a)$
denotes the group of rational characters of $M_a$. Let $\fa_{M_a} = \Hom_\ZZ (\Rat (M_a),\RR)$.  As $Q_a$ is a maximal
parabolic subgroup, $\fa_{M_a}^*\isom \RR$ and we identify $\fa_{M_a,\CC}^*$ with $\CC$ via
$s\mapsto s (\frac{\dim X +a +1}{2})^{-1}\rho_{Q_a}$, following \cite{MR2683009}, where $\rho_{Q_a}$ is the half sum of the
positive roots in $N_a$. Thus we may regard $\rho_{Q_a}$ as the number $(\dim X +a +1) /2$. In fact, it is clear that $\rho_{Q_a} (m (x,h)) = |\det x|_\A^{(\dim X +a +1) /2}$ for $m (x,h)\in M_a (\AA)$.

Let $H_a$ be the homomorphism $M_a (\AA)\rightarrow \fa_{M_a}$ such that for $m\in M_a (\AA)$ and $\xi\in \fa_{M_a}^*$
we have $\exp (\form{H_a (m)}{\xi}) = \prod_v |\xi (m_v)|_v$. We may think of $H_a (m (x,h)) \in \fa_{M_a}\isom \RR$ as
$\log (|\det x|_\AA)$. We extend $H_a$ to $G (X_a) (\AA)$ via the Iwasawa decomposition and then to
$\til {G} (X_a) (\AA)$ via projection.

Let $\chi_\psi$ be the genuine character of $\til {\GL}_1 (F_v)$ defined by
\begin{equation*}
  \chi_\psi ((g,\zeta))=\zeta \gamma (g, \psi_{1/2})^{-1}
\end{equation*}
where $\gamma (\cdot,\psi_{1/2})$, valued in the 4-th roots of unity, is defined via the Weil index. We note that the notation here agrees with that in \cite [page~521] {MR3166215}. Then via the
determinant map, we get a genuine character of $\til {\GL}_a (F_v)$ which we also denote by $\chi_\psi$. 

Let $\sigma$ be a genuine cuspidal automorphic representation of $\til {G} (X) (\AA)$. Let $\chi$ be a quadratic  automorphic
character of $\GL_a (\AA)$. Then $\chi\chi_\psi:(g,\zeta)\mapsto \chi (g)\chi_\psi ((g,\zeta))$ gives a genuine
automorphic character of $\til {\GL}_a (\AA)$. Let $\cA_{a,\psi} (s,\chi,\sigma)$ denote the space of $\CC$-valued
smooth functions $f$ on $N_a (\AA)M_a (F)\lmod \til {G} (X_a) (\AA)$ that satisfy the following properties:
\begin{enumerate}
\item $f$ is right $\til {K}_a$-finite;
\item for any $(x,\zeta)\in \til {\GL}_a (\AA)$ and $g\in \til {G} (X_a) (\AA)$ we have
  \begin{equation*}
    f (\til {m} (x,I_X,\zeta) g) =\chi\chi_\psi ((x,\zeta))|\det (x)|_\AA^{s+\rho_{Q_a}} f (g);
  \end{equation*}
\item for any fixed $k\in \til {K}_a$, the function $h\mapsto f (\til {m} (I_a,h,\zeta)k)$ on $\til {G} (X) (\AA)$ is a smooth right
  $\til {K}_{ G (X)}$-finite vector in the space of $\sigma$.
\end{enumerate}
Let $f\in\cA_{a,\psi} (0,\chi,\sigma)$ and $s\in \CC$. Since the metaplectic group splits over $G (X_a) (F)$, we can form the Eisenstein series on the metaplectic group as
in the case of non-cover groups:
\begin{equation*}
  E (g,s,f) = E (g,f_s)=\sum_{\gamma\in Q_a (F)\lmod G (X_a) (F)} f_s (\gamma g)
\end{equation*}
where $f_s (g) = \exp (\form{H_a (g)}{s})f (g)\in \cA_{a,\psi} (s,\chi,\sigma)$. 
By the general theory of Eisenstein series \cite[IV.1]{MR1361168}, it is absolute convergent for $\Re s > \rho_{Q_a}$ and has meromorphic
continuation to the whole $s$-plane with finitely many poles in the half plane $\Re s >0$, which are all real as we identify $\fa_{M_a}^*$ with $\RR$.

Let $\cP_{a,\psi} (\chi,\sigma)$ denote the set of positive poles of $E (g,s,f)$ for $f$ running over
$\cA_{a,\psi} (0,\chi,\sigma)$.
\begin{prop}\label{prop:P1-max-pole-2-Pa-max-pole}
  Assume that $\cP_{1,\psi} (\chi,\sigma)$ is non-empty and let $s_0$ be its maximal member. Then for all integers $a\ge
  1$, $s=s_0 + \frac{1}{2} (a-1)$ lies in $\cP_{a,\psi} (\chi,\sigma)$ and is its maximal member.
\end{prop}
\begin{proof}
  We remark how to carry over the proofs in \cite {MR1473165} which treated orthogonal/symplectic case and
  \cite{MR3435720} which treated the unitary case. The proofs relied on studying constant terms of the Eisenstein
  series, which are  integrals over unipotent groups. Since the metaplectic group splits over the unipotent groups,
  the proofs carry over. We just need to replace occurrences of $\chi$ in \cite [Prop.~2.1] {MR3435720} with the genuine
  character $\chi\chi_\psi$.
\end{proof}
\begin{prop}\label{prop:L-pole-2-Eis-pole}
   Let $S$ be the set of places of $F$ that contains the archimedean places and outside of which
  $\psi$, $\chi$ and $\sigma$ are unramified. Assume one of the following.
  \begin{itemize}
  \item The partial $L$-function $L^S_\psi (s,\sigma\times\chi)$ has a pole at
    $s=s_0 >\frac{1}{2}$ and that it is holomorphic for $\Re s>s_0$.
  \item The partial $L$-function $L^S_\psi (s,\sigma\times\chi)$ is non-vanishing at
    $s=s_0=\frac{1}{2}$ and is holomorphic for $\Re s > \frac{1}{2}$.
  \end{itemize}
Then for all integers $a\ge 1$, $s=s_0 +\frac{1}{2} (a-1) \in \cP_{a,\psi} (\sigma,\chi)$.
\end{prop}
\begin{proof}
  By Langlands' theory of Eisenstein series, the poles are determined by those of its constant
  terms.  In our case, this amounts to studying the poles of the intertwining operator $M (w,s)$
  attached to the longest Weyl element $w$ in the Bruhat decomposition of $Q_a (F) \lmod G (X_a) (F) /
  Q_a (F)$. Those attached to shorter Weyl elements will not be able to cancel out the pole at
  $s=s_0 +\half (a-1)$. For $v\not\in S$, the local intertwining operator $M_v (w,s)$ sends the normalised spherical vector of the local component at $v$ of $\cA_{a,\psi}(s,\chi,\sigma)$
  to  a multiple of the normalised spherical vector in another induced representation.  The
  Gindikin-Karpelevich formula extended to the  Brylinski-Deligne extensions by \cite{MR3749191} shows that the
  ratio is
  \begin{equation*}
    \prod_{1\le j \le a} \frac{L^S_\psi (s-\half (a-1)+j-1, \sigma\times\chi)}{L^S_\psi (s-\half (a-1)+j,
      \sigma\times\chi)}
    \cdot \prod_{1\le i \le j \le a}\frac{\zeta^S (2s- (a-1) +i + j-2)}{\zeta^S (2s- (a-1) +i + j-1)}
  \end{equation*}
  where $\zeta^S$ is the partial Dedekind zeta function. A little simplification shows that it is
  equal to $I_1\cdot I_2$ with
  \begin{equation*}
    I_1 = \frac{L^S_\psi (s-\half (a-1), \sigma\times\chi)}{L^S_\psi (s+\half (a+1),
      \sigma\times\chi)},
    \qquad
    I_2 = \prod_{1\le  j \le a}\frac{\zeta^S (2s- (a-1) + j-1)}{\zeta^S (2s- (a-1) + 2j -1)}.
  \end{equation*}
  If we assume the first condition, then the numerator of $I_1$ has a pole at $s=s_0 + \half (a-1)$ which cannot be cancelled
  out by other terms. If we assume the second condition, then $\zeta^S (2s- (a-1))$ has a pole $s=s_0 + \half (a-1)$ which
  cannot be cancelled out by other terms. Also note that at ramified places the local intertwining
  operators are non-zero. Thus we have shown that $s=s_0 +\frac{1}{2} (a-1) \in \cP_{a,\psi} (\sigma,\chi)$.
\end{proof}

Next we relate the Eisenstein series $E (s,g,f) = E (g, f_s)$ to Siegel Eisenstein series on the `doubled group' to glean more
information on the locations of poles.

Let $X'$ be the symplectic space with the same underlying vector space as $X$ but with the negative
symplectic form $\form {\ } {\ }_{X'} = -\form {\ } {\ }_{X}$. Let $W= X \oplus X'$ and form
$W_a = \ell_a^+ \oplus W \oplus \ell_a^-$. Let $X^\Delta = \{(x,x)\in W | x\in X\}$ and
$X^\nabla = \{(x,-x)\in W | x\in X\}$. Then $W_a$ has the polarisation
$(X^\Delta \oplus \ell_a^+) \oplus (X^\nabla \oplus \ell_a^-)$. Let $P_a$ be the Siegel parabolic
subgroup of $G (W_a)$ that stabilises $X^\nabla \oplus \ell_a^-$.  For a
$\til {K}_{G (W_a)}$-finite section $\FF_s$ in
$\Ind_{\til{P}_a (\A)}^{\til {G} (W_a) (\A)}\chi\chi_\psi|\ |_\A^s$, we form the Siegel Eisenstein
series $E^{P_a} (\cdot, \FF_s )$.

Since $G (X_a)$ acts on $\ell_a^+\oplus X \oplus \ell_a^-$ and $G (X)$ acts on $X'$, we have the obvious embedding
$\iota: G (X_a) (\A) \times G (X) (\A) \rightarrow G (W_a) (\A)$ which induces a homomorphism
$\til {\iota}: \til{G} (X_a) (\A) \times_{\mu_2} \til{G} (X) (\A) \rightarrow \til{G} (W_a) (\A)$. The cocycles for the
cover groups are compatible by \cite [Theorem~4.1] {MR1197062}.

The following is the metaplectic version of  \cite [Proposition~2.3] {MR3805648} whose proof generalises
immediately. We note that the integrand is non-genuine, and thus descends to a function on $G (X) (\AA)$.
\begin{prop}\label{prop:rel-eis-siegel-eis}
  Let $\FF_s$ be a $\til {K}_{G (W_a)}$-finite section of $\Ind_{\til{P}_a (\A)}^{\til {G} (W_a) (\A)}\chi\chi_\psi|\ |_\A^s$
  and $\phi\in \sigma$. Define a function $F_{\phi,s}$ on $\til {G}(X_a) (\A)$ by
  \begin{equation}\label{eq:section-from-siegel-type}
    F_{\phi,s} (g_a) = \int_{G (X) (\A)} \FF_s (\til {\iota} (g_a,g))\phi (g) dg.
  \end{equation}
Then
\begin{enumerate}
\item It is absolutely convergent for $\Re s > (\dim X+a+1)/2$ and has meromorphic continuation to the whole $s$-plane;
\item It is a section in $\cA_{a,\psi} (s,\chi,\sigma)$;
\item The following identity holds
  \begin{equation*}
    \int_{[G (X)]} E^{P_a} (\til {\iota} (g_a,g), \FF_s)\phi (g) dg = E (g_a,F_{\phi,s}).
  \end{equation*}
\end{enumerate}
\end{prop}
Now we list the locations of possible poles of the Siegel Eisenstein series obtained in \cite{MR1174424,MR1411571}.
\begin{prop}\label{prop:loc-pole-siegel-eis}
   The poles with $\Re s >0$ of $E^{P_a} (g,\FF_s)$ are at most simple and are contained in the set
  \begin{equation*}
    \Xi_{m+a,\psi} (\chi) = \{\half (\dim X+a)-j  |    j\in\ZZ, 0\le j < \half (\dim X+a)\}.
  \end{equation*}
\end{prop}
Combining the above propositions, we get:
\begin{prop}\label{prop:possible-loc-of-poles}
  The maximal member $s_0$ of $\cP_{1,\psi} (\chi,\sigma)$ is of the form $s_0=\half (\dim X+1) - j$ for $j\in\ZZ$ such that
  $0 \le j <\half (\dim X+1)$. 
\end{prop}
\begin{proof}
  By Prop.~\ref{prop:P1-max-pole-2-Pa-max-pole}, $s_0 +\half (a-1)$ is the maximal member of
  $\cP_{a,\psi} (\chi,\sigma)$. To ensure that \eqref{eq:section-from-siegel-type} provides enough
  sections for poles of the Siegel Eisenstein series to manifest, we need to take $a$ large
  enough. See \cite [Lemma~2.6] {MR3805648} for details. Then by
  Prop.~\ref{prop:rel-eis-siegel-eis}, it is a pole of the Siegel Eisenstein series whose possible
  poles are determined in Prop.~\ref{prop:loc-pole-siegel-eis}. Thus $s_0$ must be of the form
  $s_0=\half (\dim X+1) - j$ for $j\in\ZZ$ such that $0 \le j <\half (\dim X+1)$.
\end{proof}

\section{First Occurrence and Lowest Occurrence of Theta Correspondence}
\label{sec:first-occurr-lowest}

The locations of poles of the Eisenstein series defined in Sec.~\ref{sec:eisenstein-series} are intimately related to
the invariants, called the lowest occurrences of theta lifts, attached to $\sigma \in \cA_\cusp (\til {G} (X))$. The
lowest occurrences are defined via the more familiar notion of first occurrences. In Sec.~\ref{sec:peri-eisenst-seri}, we will
derive a relation between periods of residues of Eisenstein series and first occurrences.

First we review briefly the definition of theta lift. For each automorphic additive character $\psi$ of $\AA$, there is
a Weil representation $\omega_{\psi}$ of the metaplectic group $\Mp_{2n} (\AA)$ unique up to isomorphism.  In our case,
$X/F$ is a symplectic space and $Y/F$ is a quadratic space. From them, we get the symplectic space $Y\otimes X$. We
consider the Weil representation of $\til {G} (Y\otimes X) (\AA)$ realised on the Schwartz space
$\cS ((Y\otimes X)^+ (\AA))$ where $(Y\otimes X)^+$ denotes any maximal isotropic subspace of $Y\otimes X$. Different
choices of maximal isotropic subspaces give different Schwartz spaces which are intertwined by Fourier transforms. Thus the
choice is not essential and sometimes we simply write $\cS_{X,Y} (\AA)$ for  $\cS ((Y\otimes X)^+(\AA))$.

Let $G (Y)$ denote the isometric group of the quadratic space $Y$. We have the obvious homomorphism
$G (X) (\AA) \times G (Y) (\AA) \rightarrow \Sp (Y\otimes X) (\AA)$. By \cite{MR1286835}, there is a homomorphism
$\til {G} (X) (\AA) \times G (Y) (\AA) \rightarrow \Mp (Y\otimes X) (\AA)$ that covers it. Thus we arrive at the
representation $\omega_{\psi,X,Y}$ of $\til {G} (X) (\AA) \times G (Y) (\AA)$ on $\cS_{X,Y} (\AA)$, which is also called the
Weil representation. The explicit formulae can be found in \cite{MR1411571}, for example.

For $\Phi\in\cS_{X,Y} (\AA)$, $g\in \til {G} (X) (\AA)$ and $h\in G (Y) (\AA)$, we form the theta series
 \begin{equation*}
   \theta_{\psi,X,Y} (g,h,\Phi) = \sum_{w\in (Y\otimes X)^+ (F)}\omega_{\psi,X,Y} (g,h)\Phi (w).
 \end{equation*}
It  is absolutely convergent and is  automorphic in both $\til {G} (X) (\AA)$ and $G (Y) (\AA)$. In
general, it is of moderate growth.

With this we define the global theta lift. Let $\sigma\in\cA_\cusp (\til {G} (X))$. For $\phi\in\sigma$ and
$\Phi\in\cS_{X,Y} (\AA)$, define
\begin{equation*}
  \theta_{\psi,X}^Y (h,\phi,\Phi) = \int_{[G (X)]}\theta_{\psi,X,Y} (g,h,\Phi)\overline {\phi (g)} dg.
\end{equation*}
The integrand is the product of two genuine functions on $\til {G} (X) (\AA)$ and hence can be regarded as a
function on $G (X) (\AA)$. 
 We define the global theta lift $\theta_{\psi,X}^Y (\sigma)$ of $\sigma$ to be the space spanned by the above
 integrals. Let $\chi_Y$ be the quadratic character of $\AA^\times$ associated to $Y$ as in \cite[(0.7)]{MR1289491}.

 \begin{defn}
   \begin{enumerate}
   \item The first occurrence index $\FO_\psi^Y (\sigma)$ of $\sigma$ in the Witt tower of $Y$ is defined to be
     \begin{equation*}
       \min_{Y'}  \{\dim Y' | \theta_{\psi , X}^{Y'} (\sigma) \neq 0  \},
     \end{equation*}
     where $Y'$ runs through all the quadratic spaces in the same Witt tower as $Y$.
   \item  For a quadratic automorphic character $\chi$ of $\AA^\times$, define the
lowest occurrence index $\LO_{\psi, \chi} (\sigma)$ with respect to $\chi$, to be
\begin{equation*}
   \min_Y\{\FO_\psi^Y (\sigma)\}
\end{equation*}
where $Y$ runs over all quadratic spaces with $\chi_Y = \chi$ and odd dimension.
\end{enumerate}

\end{defn}
With this we get  results on the lowest occurrence and the location of the maximal pole of the Eisenstein series. We will strengthen the results using periods in Sec.~\ref{sec:peri-eisenst-seri}.
\begin{thm}\label{thm:Eis-pole-2-bound-on-LO}
  Let $\sigma\in \cA_\cusp (\til {G} (X))$ and $s_0$ be the maximal member of $\cP_{1,\psi} (\chi,\sigma)$. Then
  \begin{enumerate}
  \item $s_0=\half (\dim X +1) - j$ for some $j\in\ZZ$ such that $0\le j < \half (\dim X+1)$;
  \item $\LO_{\psi,\chi} (\sigma) \le 2j+1$;
  \item $2j+1 \ge r_X$ where $r_X = \dim X/2$ is the Witt index of $X$.
  \end{enumerate}
\end{thm}
\begin{proof}
  The third part is stated in terms of the Witt index of $X$ so that the analogy with \cite [Thm.~3.1] {MR3435720} is more obvious.
  
  The first part is just Prop.~\ref{prop:possible-loc-of-poles} which we reproduce here.

  For the second part we make use of the regularised Siegel-Weil formula that links residues of
  Siegel Eisenstein series to theta integrals. Assume that $s_0=\half (\dim X+1) -j $ is the maximal
  element in $\cP_{1,\psi} (\chi,\sigma)$. By Prop.~\ref{prop:P1-max-pole-2-Pa-max-pole},
  $s_0+\half (a-1)=\half (\dim X+a)-j$ is the maximal element in $\cP_{a,\psi} (\chi,\sigma)$. In
  other words, $E^{Q_a} (g,f_s)$ has a pole at $s=\half (\dim X+a)-j$ for some $f$. (We put in a
  superscript to be clear that we are talking about the Eisenstein series associated to the
  parabolic subgroup $Q_a$.) Furthermore if $a$ is large enough this $f_s$ can be taken to be of the
  form $F_{\phi,s}$ (see Prop.~\ref{prop:rel-eis-siegel-eis} and also the proof of Prop.~\ref{prop:possible-loc-of-poles}).  Enlarge $a$ if necessary to ensure
  that $2j < \dim X+a$. Then the regularised Siegel-Weil formula shows that
  \begin{equation*}
    \Res_{s=\half (\dim X+a)-j}E^{P_a} (\til {\iota} (g_a,g), \FF_s) = \sum_{Y}  \int_{[G (Y)]} \theta_{\psi,W_a,Y} (\til {\iota}
    (g_a,g),h,\Phi_Y) dh
  \end{equation*}
where the sum runs over quadratic spaces $Y$ of dimension $2j+1$ with $\chi_Y =\chi$ and $\Phi_Y \in \cS_{W_a,Y} (\AA)$
are appropriate $\til {K}_{G (W_a)}$-finite Schwartz functions so that the theta integrals are absolutely convergent. We remark that for any
Schwartz function there is a natural way to `truncate' it for regularising the theta integrals. See \cite{MR1289491}
and \cite{MR1863861} for details. Now we integrate both sides against $\phi (g)$ over $[G (X)]$. Then by Prop.~\ref{prop:rel-eis-siegel-eis} the left hand side
becomes a residue of $E^{Q_a} (g_a,s,F_{s,\phi})$, which is non-vanishing for some choice of $\phi\in\sigma$. Thus at
least one term on the right hand side must be non-zero. In other words, there exists a quadratic space $Y$ of dimension
$2j+1$ and of character $\chi$ such that
\begin{equation*}
  \int_{[G (X)]} \int_{[G (Y)]}\theta_{\psi,W_a,Y} (\til {\iota} (g_a,g),h,\Phi_Y)\phi (g) dh dg \neq 0.
\end{equation*}
We may replace $\Phi_Y$ by $\Phi_Y^{(1)}\otimes \Phi_Y^{(2)}$ with $\Phi_Y^{(1)}\in\cS_{X_a,Y} (\AA)$ and
$\Phi_Y^{(2)}\in\cS_{X,Y} (\AA)$ such that the inequality still holds. Then after separating the variables we get
\begin{equation*}
  \int_{[G (Y)]} \theta_{\psi,X_a,Y} ( g_a,h,\Phi_Y^{(1)}) \int_{[G (X)]} \overline{\theta_{\psi,X,Y} (g,h,\Phi_Y^{(2)})}\phi (g) dg dh \neq 0.
\end{equation*}
The complex conjugate appears because we have switched from $X'$ to $X$.  In particular, the inner integral is
non-vanishing, but it is exactly the complex conjugate of the theta lift of $\phi$ to $G (Y)$. Thus the lowest
occurrence $\LO_{\psi,\chi} (\sigma)$ must be lower or equal to $2j+1$.

For the third part, assume that the lowest occurrence is realised by $Y$. Thus $\dim Y = 2j' +1$ with
$j'\le j$. Let $\pi = \theta_{\psi,X}^Y (\sigma)$.  It is
non-zero and it is cuspidal because it is the first occurrence representation in the Witt tower of $Y$. It is also
irreducible by \cite [Theorem~1.3] {MR2330445}.  By stability, $\FO_{\psi,Y}^X (\pi) \le 2 (2j'+1)$. By the
involutive property in \cite [Theorem~1.3] {MR2330445}, we have
\begin{equation*}
  \theta_{\psi,Y}^X (\pi) =   \theta_{\psi,Y}^X (\theta_{\psi,X}^Y (\sigma)) = \sigma.
\end{equation*}
Note in \cite {MR2330445}, the theta lift was defined without the complex conjugation on the theta series, which resulted in
the use of $\psi^{-1}$ in the inner theta lift. Thus we find $\dim X \le 2 (2j'+1) \le 2 (2j+1)$ and we get $2j+1 \ge r_X$.  
\end{proof}
Thus we glean some information on poles of $L^S_\psi (s,\sigma\times\chi)$  (or equivalently factors of the form $(\chi,b)$
in the Arthur parameter of $\sigma$) and lowest occurrence of theta lift.
\begin{thm}
  Let $\sigma\in\cA_\cusp (\til {G} (X))$ and $\chi$ be a quadratic automorphic character of $\AA^\times$. Then the following hold.
  \begin{enumerate}
  \item \label{item:2}
    If the partial $L$-function $L^S_\psi (s,\sigma\times\chi)$ has a pole at $s=\half (\dim X+1)-j>0$, then $j$ is an integer such that $\frac{1}{4}\dim X -\half \le j < \half \dim X$. In other words, the possible positive poles of the partial $L$-function $L^S_\psi (s,\sigma\times\chi)$ are non-integral half-integers $s_0$ such that $\half < s_0 \le \frac{1}{4}\dim X +1$.
  \item \label{item:1}If the partial $L$-function $L^S_\psi (s,\sigma\times\chi)$ has a pole at $s=\half (\dim X+1)-j>0$, then $\LO_{\psi,\chi}
    (\sigma) \le 2j+1$.
  \item If $\LO_{\psi,\chi} (\sigma) = 2j +1 < \dim X+2$, then $L^S_\psi (s,\sigma\times\chi)$ is holomorphic for $\Re s> \half
    (\dim X+1)-j$.
  \item If $\LO_{\psi,\chi} (\sigma) = 2j +1 > \dim X+2$, then $L^S_\psi (s,\sigma\times\chi)$ is holomorphic for $\Re s> \half$.
  \end{enumerate}
\end{thm}
\begin{proof}
By \cite [Prop.~6.2] {MR3595433}, the poles $s$ of $L_\psi^S (s,\sigma\times\chi)$ with $\Re s \ge \half$ are contained in the set $\{\frac{3}{2},
  \frac{5}{2},\ldots , \frac{\dim X+1}{2}\}$, where the extraneous $+\half$ should be
  removed from $L (s+\half,\pi\times\chi,\psi)$ in the statement there.   The proof there actually implies that this is the set of  possible poles in $\Re s >0$. See also  \cite[Theorem~9.1]{MR3211043}.
   Let $s=\half (\dim X+1)-j_0$ be its maximal pole. Then $0 \le j_0 <\half \dim X$. By Prop.~\ref{prop:L-pole-2-Eis-pole}, $s=\half (\dim X+1)-j_0$ is in $\cP_{1,\psi} (\chi,\sigma)$. Assume $\half (\dim X+1)-j_0'$ is the maximal member in $\cP_{1,\psi} (\chi,\sigma)$. Then by Thm.~\ref{thm:Eis-pole-2-bound-on-LO}, we have $\frac{1}{4}\dim X - \half \le j_0'<\half(\dim X +1)$. As $j_0'\le j_0$, we get $\frac{1}{4}\dim X - \half \le j_0$. We have shown part~\eqref{item:2}.

  Assume that $L^S_\psi (s,\sigma\times\chi)$ has
  a pole at $s=\half (\dim X+1)-j>0$. Let $s=\half (\dim X+1)-j'$ be its maximal pole. Then by Prop.~\ref{prop:L-pole-2-Eis-pole},
  $s=\half (\dim X+1)-j'$ is in $\cP_{1,\psi} (\chi,\sigma)$. Let $s=\half (\dim X+1)-j''$ be the maximal member in $\cP_{1,\psi}
  (\chi,\sigma)$. By Thm.~\ref{thm:Eis-pole-2-bound-on-LO}, $\LO_{\psi,\chi} (\sigma) \le 2j''+1\le 2j'+1 \le 2j+1$.

  Assume that $\LO_{\psi,\chi} (\sigma) = 2j +1 < \dim X+2$ and $L^S_\psi (s,\sigma\times\chi)$ has a pole in
  $\Re s> \half (\dim X+1)-j$, say at $s= \half (\dim X+1) - j_0$ for $j_0< j$. Then by part~\eqref{item:1},
  $\LO_{\psi,\chi} (\sigma) \le 2j_0 +1 < 2j+1$. We arrive at a contradiction.

  Assume that $\LO_{\psi,\chi} (\sigma) = 2j +1 > \dim X+2$ and $L^S_\psi (s,\sigma\times\chi)$ has a pole in
  $\Re s> \half$, say at $s= \half (\dim X+1) - j_0$ for $j_0< \half (\dim X+1)$. Then by part~\eqref{item:1},
  $\LO_{\psi,\chi} (\sigma) \le 2j_0 +1 < \dim X+2$. We arrive at a contradiction.
  \end{proof}

In terms of A-parameters, the theorem above says:
 \begin{thm}\label{thm:chi-b-factor-LO}
   Let $\sigma\in\cA_\cusp (\til {G} (X))$ and $\chi$ be a quadratic automorphic character of $\AA^\times$. Let $\phi_\psi (\sigma)$ denote
   the A-parameter of $\sigma$ with respect to $\psi$. Then the following hold.
   \begin{enumerate}
   \item \label{item:3}If $\phi_\psi (\sigma)$ has a $(\chi,b)$-factor, then $b\le \half\dim X +1$.
  \item If $\phi_\psi (\sigma)$ has a $(\chi,b)$-factor  with $b$ maximal among all simple factors of $\phi_\psi (\sigma)$, then  $\LO_{\psi,\chi}    (\sigma) \le \dim X - b +1$.
  \item If $\LO_{\psi,\chi} (\sigma) = 2j +1 < \dim X+2$, then $\phi_\psi (\sigma)$ cannot have a $(\chi, b)$-factor with $b$ maximal among all simple factors of $\phi_\psi (\sigma)$ and $b>\dim
    X -2j$.
  \item If $\LO_{\psi,\chi} (\sigma) = 2j +1 > \dim X+2$, then  $\phi_\psi (\sigma)$ cannot have a $(\chi, b)$-factor with $b$ maximal among all simple factors of $\phi_\psi (\sigma)$.
  \end{enumerate}
\end{thm}

\begin{rmk}
  Part~\eqref{item:3} of Thm.~\ref{thm:chi-b-factor-LO} implies a similar statement in the metaplectic case to the symplectic result in \cite[Theorem~3.1, Remark~3.3]{MR3969876} which are a key input to results in \cite[Sec.~5]{MR3969876} on the generalised Ramanujan problem.
\end{rmk}

\section{Fourier coefficients of theta lift and periods}
\label{sec:four-coeff-theta}

In this section we compute certain Fourier coefficients of theta lifts from the metaplectic group $\til {G} (X)$ to odd
orthogonal groups and we derive some vanishing and non-vanishing results of period integrals over symplectic subgroups and Jacobi subgroups of
$G (X)$.

We define these subgroups first. Let $Z$ be a non-degenerate symplectic subspace of $X$ and let $L$ be a totally
isotropic subspace of $Z$. Then $G (Z)$ is a symplectic subgroup of $G (X)$. Let $Q (Z,L)$ be the parabolic subgroup of
$G (Z)$ that stabilisers $L$ and $J (Z,L)$ be the Jacobi subgroup that fixes $L$ element-wise. When $L=0$, then
$J (Z,L)$ is just $G (Z)$.  In this section, $Z$ (resp. $Z'$, $Z''$) will always be a non-degenerate symplectic subspace of
$X$ and $L$ (resp. $L',L''$) will always be a totally isotropic subspace of $Z$ (resp. $Z'$, $Z''$). For concreteness, we define
distinction.

\begin{defn} Let $G$ be a reductive group and $J$ be a subgroup of $G
$. Let $\sigma$ be an automorphic representation of $G$. For
$f\in\sigma$, if the period integral
\begin{equation} \int_{[J]} f (g)dg.
    \end{equation} is absolutely convergent and non-vanishing, we
say that $f$ is $J$-distinguished.  Assume that for all $f\in \sigma$,
the period integral is absolutely convergent.  If there exists $f\in\sigma$ such that $f$ is
$J$-distinguished, then we say $\sigma$ is
$J$-distinguished. 
\end{defn}

Let $Y$ be an anisotropic quadratic space, so that it sits at the bottom of its Witt tower. We can form the augmented
quadratic spaces $Y_r$ analogously by adjoining $r$-copies of the hyperbolic plane to $Y$. Let $\Theta_{\psi,X,Y_r} $
denote the space of theta functions $\theta_{\psi,X,Y_r} (\cdot, 1, \Phi)$ for $\Phi$ running over $\cS_{X,Y_r}
(\AA)$. Then we have the following that is derived from the computation of Fourier coefficients of theta lift of $\sigma$ to
various $G (Y_r)$'s.

\begin{prop}\label{prop:FO-2-distinction}
  Let $\sigma \in \cA _\cusp (\til {G} (X))$ and let $Y$ be an odd dimensional anisotropic quadratic space. Assume that
  $\FO_\psi^Y (\sigma) = \dim Y + 2 r_0$. Let $Z$ be a non-degenerate symplectic subspace of $X$ and $L$ a totally isotropic
  subspace of $Z$. Let $r$ be a non-negative integer. Then the following hold.
  \begin{enumerate}
  \item If $\dim X - \dim Z + \dim L + r < r_0$, then $\sigma \otimes \overline {\Theta_{\psi,X,Y_r}}$ is not $J (Z,L)$-distinguished.
  \item If $\dim X - \dim Z + \dim L = r_0$ and $\dim L = 0, 1$, then $\sigma \otimes \overline {\Theta_{\psi,X,Y}}$ is $J
    (Z,L)$-distinguished for some choice of $(Z,L)$.
  \end{enumerate}
\end{prop}
\begin{proof}
  This is the metaplectic analogue of \cite [Proposition~3.2] {MR3805648}. The period integrals involved are the complex
  conjugates of
  \begin{equation}\label{eq:period-theta}
    \int_{[J (Z,L)]}\overline {\phi (g)}\theta_{\psi,X,Y_r} (g,1,\Phi) dg
  \end{equation}
  where $\phi \in \sigma$ and $\Phi\in \cS_{X,Y_r} (\AA)$. As both factors of the integrand are genuine in $\til {G} (X)
  (\AA)$, the product is a well-defined function of $G (X) (\AA)$, which we can restrict to $J (Z,L) (\AA)$. We note
  that 
  the proof of \cite [Proposition~3.2] {MR3805648} involves taking Fourier
   coefficients of automorphic forms (on orthogonal groups) in the spaces of $\theta_{\psi,X}^{Y_t} (\sigma)$ for varying $t$
   and these Fourier coefficients can be written as sums of integrals which have
   \eqref{eq:period-theta}  as inner integrals.  The proof there also works for odd orthogonal groups.

\end{proof}

We have a converse in the following form.
\begin{prop}
  Let $\sigma\in\cA_\cusp (\til {G} (X))$ and let $Y$ be an odd dimensional anisotropic quadratic space. Assume that
  $\sigma\otimes \overline {\Theta_{\psi,X,Y}}$ is $J (Z,L)$-distinguished for some choice of $(Z,L)$ with $\dim L =0,1$, but that it is
  not $J (Z',L')$-distinguished for all $(Z',L')$ with $\dim L'=0,1$ and $\dim Z' - \dim L' > \dim Z - \dim L
$. Then if we set $r = \dim X - \dim Z + \dim L$, then $\FO_\psi^Y (\sigma) = \dim Y + 2r$.
\end{prop}
Combined with Prop.~\ref{prop:FO-2-distinction}, we have the implication.
\begin{prop}
  Let $\sigma\in\cA_\cusp (\til {G} (X))$ and let $Y$ be an odd dimensional anisotropic quadratic space. Assume that for
  some $t>0$, $\sigma\otimes\overline {\Theta_{\psi,X,Y}}$ is not $J (Z',L')$-distinguished for all $(Z',L')$ with $\dim L' =0,1$
  and $\dim Z' - \dim L' >t$. Then $\sigma\otimes\overline {\Theta_{\psi,X,Y}}$ is not $J (Z'',L'')$-distinguished for all $(Z'',L'')$
  with  $\dim Z'' - \dim L'' >t$.
\end{prop}
Both  propositions can be proved in the same way as in \cite[Propositions~3.4, 3.5]{MR3805648}.

\section{Periods of Eisenstein series}
\label{sec:peri-eisenst-seri}

In this section, to avoid clutter we write $G$ for the $F$-rational points $G (F)$ of $G$ etc. We also suppress the
additive character $\psi$ from notation. In this section, again $Z$ (resp. $Z'$, $Z''$) will always be a non-degenerate
symplectic subspace of $X$ and $L$ (resp. $L',L''$) will always be an isotropic subspace of $Z$ (resp. $Z'$, $Z''$). The
periods considered in this section involve an Eisenstein series and a theta function. They are supposed to run parallel
to those in  Sec.~\ref{sec:four-coeff-theta} which involve a cuspidal automorphic form and a theta function \eqref{eq:period-theta}. Through our computation, we will relate the periods of Eisenstein series to \eqref{eq:period-theta} and hence also to invariants of theta correspondence.

Assume that $Y$ is an anisotropic quadratic space of odd dimension. Let $Z$ be a non-degenerate symplectic subspace of
$X$. Let $V$ be its orthogonal complement in $X$ so that $X=V \perp Z$. Form similarly the augmented space
$Z_1=\ell_1^+ \oplus Z \oplus \ell_1^- \subset X_1$. Then $X_1 =  V \perp Z_1$. Let $L$ be a totally isotropic
subspace of $Z$ of dimension $0$ or $1$. Then the subgroup $J (Z_1,L)$ of $G (Z_1)$ is defined to be the fixator of
$L$ element-wise. Via the natural embedding of $G (Z_1)$ into $G (X_1)$, we regard $J (Z_1,L)$ as a subgroup of $G(X_1)$.

Then the period integrals we consider are of the form
\begin{equation*}
  \int_{[J (Z_1,L)]} E (g,s,f)\overline {\theta_{\psi,X_1,Y} (g,1,\Phi)} dg
\end{equation*}
for $f_s\in \cA_1 (s,\chi,\sigma)$ and $\Phi\in\cS_{X_1,Y} (\AA)$.  The integrals diverge in general, but they
can be regularised by using the Arthur truncation \cite{MR558260,MR518111} (See also \cite [I.2.13] {MR1361168} which
includes the metaplectic case). We write down how the Arthur truncation works in the present case. We follow
closely Arthur's notation. We note that the need for Jacobi groups is special to the symplectic/metaplectic case
due to the lack of `odd symplectic groups'. Compare with the unitary case \cite{wu21:chi-b-unitary_ii} and the orthogonal
case \cite{MR2540878}. 

We continue to use our notation from Sec.~\ref{sec:eisenstein-series}. Note that the Eisenstein series is attached to a maximal parabolic subgroup $Q_1$ of $G (X_1)$. In this case, $\fa_{M_1}$ is
1-dimensional and we identify it with $\RR$. For $c\in \RR_{>0}$, set $\hat {\tau}^c$ to be the characteristic function
of $\RR_{> \log c}$ and set $\hat {\tau}_c = 1_\RR - \hat {\tau}^c$. Then the truncated Eisenstein series is
\begin{equation}\label{eq:truncation}
  \Lambda^c E (g,s,f) = E (g,s,f) - \sum_{\gamma \in Q_1 \lmod G (X_1)} E_{Q_1} (\gamma g , s, f) \hat {\tau}^c (H
  (\gamma g))
\end{equation}
where $E_{Q_1} (\cdot, s, f)$ is the constant term of $E (\cdot,s,f)$ along $Q_1$. The summation has only finitely many
non-vanishing terms for each fixed $g$. The truncated Eisenstein series is
rapidly decreasing while the theta series is of moderate growth. Thus if we replace the Eisenstein series with the truncated Eisenstein
series, we get a period integral that is absolutely convergent:
\begin{equation}\label{eq:period-eis-theta}
  \int_{[J (Z_1,L)]} \Lambda^c E (g,s,f)\overline {\theta_{\psi,X_1,Y} (g,1,\Phi)} dg.
\end{equation}

We state our main theorems.

\begin{thm}\label{thm:distinction-2-eis-pole}
  Let $\sigma\in\cA_\cusp (\til {G})$ and let $Y$ be an anisotropic quadratic space of odd dimension. Assume that
  $\sigma\otimes\overline {\Theta_{\psi,X,Y}}$ is $J (Z,L)$-distinguished for some $(Z,L)$ with $\dim L=0$ or $1$, but it is not
  $J (Z',L')$-distinguished for all $(Z',L')$ such that $\dim L' = 0,\ldots , \dim L +1$ and
  $\dim Z' - \dim L' > \dim Z -\dim L$. Set $r= \dim X - \dim Z + \dim L$ and $s_0 = \half (\dim X - (\dim Y +2r )
  +2)$. Then $E (g,s,f)$ has a pole at $s=s_0$ for some choice of $f_s \in \cA_{1,\psi} (s,\chi_Y,\sigma)$.
\end{thm}

For a complex number $s_1$, let $\cE_{s_1} (g,f_s)$ denote the residue of $E (g,f_s)$ at $s=s_1$. When $s_1>0$, we
consider the period integral
\begin{equation}
  \label{eq:period-res-eis-theta}
  \int_{[J (Z_1,L)]} \cE_{s_1} (g,f_s)\overline {\theta_{\psi,X_1,Y} (g,1,\Phi)} dg.
\end{equation}
We have the following result.

\begin{thm}\label{thm:period-res-eis-theta}
  Adopt the same setup as in Thm.~\ref{thm:distinction-2-eis-pole}. Assume further that $s_0 > 0$. Then the following hold.
  \begin{enumerate}
  \item $s_0$ lies in $\cP_{1,\psi} (\sigma,\chi)$ and it corresponds to  a simple pole.
  \item $\cE_{s_0} (g,f_s)\overline {\theta_{\psi,X_1,Y} (g,1,\Phi)}$ is $J (Z_1,L)$-distinguished for some choice of
    $f_s\in\cA_{1,\psi} (s,\chi_Y,\sigma)$ and $\Phi\in\cS_{X_1,Y} (\AA)$.
  \item If $0<s_1\neq s_0$, then $\cE_{s_1} (g,f_s)\overline {\theta_{\psi,X_1,Y} (g,1,\Phi)}$ is not
    $J (Z_1,L)$-distinguished for any $f_s\in\cA_{1,\psi} (s,\chi_Y,\sigma)$ and $\Phi\in\cS_{X_1,Y} (\AA)$.
  \item For $(Z'',L'')$ with $\dim L'' = 0$ or $1$ such that $\dim Z'' - \dim L'' > \dim Z - \dim L$, then for any $s_1>0$,
    $\cE_{s_1} (g,f_s)\overline {\theta_{\psi,X_1,Y} (g,1,\Phi)}$ is not $J (Z''_1,L'')$-distinguished for any
    $f_s\in\cA_{1,\psi} (s,\chi_Y,\sigma)$ and $\Phi\in\cS_{X_1,Y} (\AA)$.
  \end{enumerate}
\end{thm}

Now we bring in the first occurrence index for $\sigma$. By Prop.~\ref{prop:FO-2-distinction}, if
$\FO_{\psi}^Y (\sigma)= \dim Y+2r$, then $\sigma\otimes\overline {\Theta_{\psi,X,Y}}$ is $J (Z,L)$ distinguished for
some $(Z,L)$ such that $\dim X - \dim Z + \dim L = r$  and $\dim L \equiv r \pmod 2$ with $\dim L=0$ or $1$ and it is not
$J (Z',L')$-distinguished for any $(Z',L')$ such that $\dim X - \dim Z' + \dim L' < r $  or equivalently $\dim Z' - \dim
L' > \dim Z -\dim L$. In other words the conditions of Thm.~\ref{thm:distinction-2-eis-pole} are satisfied. We get:
\begin{cor}\label{cor:FO-2-distinction-res-eis-theta}
  Let $\sigma\in\cA_\cusp (\til {G})$ and let $Y$ be an anisotropic quadratic space of odd dimension. Assume that
  $\FO_{\psi}^Y (\sigma)= \dim Y+2r$. Let $s_0 = \half (\dim X - (\dim Y +2r ) +2)$. Then $E (g,s,f)$ has a pole at
  $s=s_0$ for some choice of $f_s \in \cA_{1,\psi} (s,\chi_Y,\sigma)$. Assume further that $s_0>0$. Then the following hold.
  \begin{enumerate}
  \item $s_0$ lies in $\cP_{1,\psi} (\sigma,\chi)$ and it corresponds to  a simple pole.
  \item There exist a non-degenerate symplectic subspace $Z$ of $X$ and an isotropic subspace $L$ of $Z$ satisfying $\dim X - \dim Z + \dim L = r$ and $\dim L \equiv r \pmod {2}$ with $\dim L=0$ or $1$ such that $\cE_{s_0} (g,f_s)\overline {\theta_{\psi,X_1,Y} (g,1,\Phi)}$ is $J (Z_1,L)$-distinguished for some choice of
    $f_s\in\cA_{1,\psi} (s,\chi_Y,\sigma)$ and $\Phi\in\cS_{X_1,Y} (\AA)$.
  \item If $0<s_1\neq s_0$, then $\cE_{s_1} (g,f_s)\overline {\theta_{\psi,X_1,Y} (g,1,\Phi)}$ is not
    $J (Z_1,L)$-distinguished for any $f_s\in\cA_{1,\psi} (s,\chi_Y,\sigma)$ and $\Phi\in\cS_{X_1,Y} (\AA)$ where $Z$ is any non-degenerate symplectic subspace of $X$ and $L$ is any isotropic subspace of $Z$
 such that $\dim X - \dim Z + \dim L = r$ and $\dim L \equiv r \pmod {2}$ with $\dim L=0$ or $1$.
  \item For $(Z'',L'')$ with $\dim L'' = 0$ or $1$ such that $\dim Z'' - \dim L'' > \dim Z - \dim L$ and for $s_1>0$,
    $\cE_{s_1} (g,f_s)\overline {\theta_{\psi,X_1,Y} (g,1,\Phi)}$ is not $J (Z''_1,L'')$-distinguished for any
    $f_s\in\cA_{1,\psi} (s,\chi_Y,\sigma)$ and $\Phi\in\cS_{X_1,Y} (\AA)$.
  \end{enumerate}
\end{cor}
With these results on period integrals involving residues of Eisenstein series, we can strengthen
Thm.~\ref{thm:Eis-pole-2-bound-on-LO}.
\begin{thm}\label{thm:strengthend-eis-pole-2-LO}
    Let $\sigma\in\cA_\cusp (\til {G} (X))$ and let $s_0$ be the maximal element in $\cP_{1,\psi} (\chi,\sigma)$. Write
    $s_0=\half (\dim X +1)-j$. Then
  \begin{enumerate}
  \item $j$ is an integer such that $\frac{1}{4} (\dim X -2) \le j < \half (\dim X +1)$.
  \item $\LO_{\psi,\chi} (\sigma) = 2j+1$.
  \end{enumerate}
\end{thm}
\begin{proof}
  Part (1) is a restatement of Thm.~\ref{thm:Eis-pole-2-bound-on-LO} and we have also shown that $\LO_{\psi,\chi}
  (\sigma) \le 2j+1$ there.  We just need to show the other inequality. Assume that the lowest occurrence of $\sigma$ is
  achieved in the Witt tower of $Y$. Thus $\LO_{\psi,\chi} (\sigma) = \FO_\psi^Y (\sigma)$. We assume $\FO_\psi^Y
  (\sigma) = 2j' +1$. Then by Cor.~\ref{cor:FO-2-distinction-res-eis-theta}, $ \half (\dim X - (2j' +1) +2)$ is a
 member of $\cP_{1,\psi} (\sigma,\chi)$. Since $s_0=\half (\dim X +1)-j$ is the maximal member, we get $\half (\dim X
 +1)-j\ge\half (\dim X - (2j' +1) +2)$ or $j'\ge j$. We have shown that $\LO_{\psi,\chi} (\sigma)\ge 2j+1$.
\end{proof}

Now we proceed to prove the main theorems. We cut the period integral \eqref{eq:period-eis-theta} into several parts and evaluate each part. We proceed formally and  justify that each part is
absolutely convergent for $\Re s$ and $c$ large and that it has meromorphic continuation to all $s\in\CC$ at the end.

We have $E_{Q_1} (\gamma g , s, f) = f_s (g) + M (w,s)f_s (g)$ as continued meromorphic functions in $s$ where $M (w,s)$ is the
intertwining operator associated to the longest Weyl element $w$ in $Q_1 (F) \lmod G (X_1)(F) /Q_1(F)$. When $\Re
(s)>\rho_{Q_1}$, we get
\begin{equation*}
  \Lambda^c E (g,s,f) = \sum_{\gamma \in Q_1 \lmod G (X_1)} f_s (\gamma g) \hat{\tau}_c (H(\gamma g)) - \sum_{\gamma \in Q_1 \lmod
                                                                                                          G (X_1)} M (w,s)f_s (\gamma g) \hat {\tau}^c (H (\gamma g)).
\end{equation*}
The first series is absolutely convergent for $\Re
(s)>\rho_{Q_1}$ while the second series has only finitely many non-vanishing terms for each fixed $g$. Both series
have meromorphic continuation to the whole complex plane. Set
\begin{align*}
  \xi_{c,s} (g) &= f_s (g) \overline {\theta_{X_1,Y} (g,1,\Phi)}\hat {\tau}_c (H (g)) ;\\
  \xi_s^c (g) &= M (w,s) f_s (g) \overline {\theta_{X_1,Y} (g,1,\Phi)}\hat {\tau}^c (H (g))
\end{align*}
and set
\begin{equation*}
  I (\xi) = \int_{[J (Z_1,L)]} \sum_{\gamma \in Q_1 \lmod G (X_1)} \xi (\gamma g) dg
\end{equation*}
for $\xi = \xi_{c,s}$ or $\xi_s^c$. Then the period integral we are concerned with is equal to
\begin{equation*}
  \int_{[J (Z_1,L)]} \Lambda^c E (g,s,f)\overline {\theta_{X_1,Y} (g,1,\Phi)}dg = I (\xi_{c,s}) - I (\xi_s^c)
\end{equation*}
as long as the integrals defining the two terms on  right hand side are absolutely convergent. 

Now we cut $I (\xi)$ into
several parts according to $J (Z_1,L)$-orbits in $Q_1 \lmod G (X_1)$, show that the parts are  absolute convergent and compute their values.
The set $Q_1 \lmod G (X_1)$ corresponds to isotropic lines in $X_1$:
\begin{align*}
  Q_1 \lmod G (X_1) &\leftrightarrow \{\text {Isotropic lines in $X_1$}\} \\
  \gamma &\leftrightarrow \ell_1^- \gamma.
\end{align*}
For an isotropic line $\ell\in X_1$, we write $\gamma_\ell$ for an element in $G (X_1)$ such that $\ell_1^- \gamma_\ell
=\ell$. We will exercise our freedom to choose $\gamma_\ell$ as simple as possible.

First we consider the $G (Z_1)$-orbits in $Q_1 \lmod G (X_1)$.  If $L$ is non-trivial, then we will further consider the $J (Z_1,L)$-orbits. Given an isotropic line in $X_1$, we pick a non-zero vector $x$ on the line and write
$x=v+z$ according to the decomposition $X_1 = V \perp Z_1$. There are three cases, so we define:
\begin{enumerate}
\item $\Omega_{0,1}$ to be the set of (isotropic) lines in $X_1$ whose projection to $V$ is $0$;
\item $\Omega_{1,0}$ to be the set of (isotropic) lines in $X_1$ whose projection to $Z_1$ is $0$;
\item $\Omega_{1,1}$ to be the set of (isotropic) lines in $X_1$ whose projections to both $V$ and $Z_1$ are non-zero.
\end{enumerate}
Each set is stable under the action of $G (Z_1)$. The set $\Omega_{0,1}$ forms one $G (Z_1)$-orbit. We may take
$Fe_1^-=\ell_1^-$ as the orbit representative. Its stabiliser in $G (Z_1)$ is the parabolic subgroup $Q (Z_1,Fe_1^-)$
that stabilises $Fe_1^-$. The set $\Omega_{1,0}$ is acted on by $G (Z_1)$ trivially. A set of orbit representatives is
$(V-\{0\})/F^\times$ and for each representative the stabiliser is $G (Z_1)$. Finally consider the set $\Omega_{1,1}$. A
line $F (v+z)$ in $\Omega_{1,1}$ is in the same $G (Z_1)$-orbit as $F (v+e_1^-)$. For $v_1,v_2\in V - \{0\}$, $F
(v_1+e_1^-)$ and $F(v_2+e_1^-)$ are in the same orbit, if and only if there exist $\gamma\in G (Z_1)$ and $c\in
F^\times$ such that $v_1 + e_1^-\gamma = c (v_2 + e_1^-)$. Thus a set of orbit representatives is $F (v+e_1^-)$ for $v$
running over a set of representatives of $(V-\{0\})/F^\times$. For each orbit representative, the stabiliser is $J
(Z_1,Fe_1^-)$.

Thus we cut the series over $Q_1\lmod G (X_1)$ into three parts and we get
\begin{equation*}
  I (\xi) = I_{\Omega_{0,1}} (\xi) + I_{\Omega_{1,0}} (\xi) + I_{\Omega_{1,1}} (\xi)
\end{equation*}
with
\begin{align}\label{eq:I-Omegas}
  I_{\Omega_{0,1}} (\xi) &= \int_{[J (Z_1,L)]} \sum_{\delta\in Q (Z_1,Fe_1^-) \lmod G (Z_1)} \xi (\delta g) dg \\
\nonumber  I_{\Omega_{1,0}} (\xi)&= \int_{[J (Z_1,L)]}  \sum_{v\in (V-\{0\})/F^\times} \xi (\gamma_{F v} g) dg \\
\nonumber  I_{\Omega_{1,1}} (\xi) &= \int_{[J (Z_1,L)]} \sum_{\delta \in J(Z_1,Fe_1^-) \lmod G (Z_1)} \sum_{v\in (V-\{0\})/F^\times} \xi (\gamma_{F (v+e_1^-)}\delta g) dg.
\end{align}

\begin{prop}\label{prop:abs-conv}
  For $\xi=\xi_{c,s}$ or $\xi_s^c$, the integrals $I_{\Omega_{0,1}} (\xi)$, $I_{\Omega_{1,0}} (\xi)$ and $I_{\Omega_{1,1}} (\xi)$ are absolutely convergent for $\Re s$ and $c$
large enough and  each has meromorphic continuation to the whole complex plane.
\end{prop}
\begin{proof}
  The whole Sec.~5 of \cite{MR3805648} deals with bounds and convergence issues for the symplectic group case. It
  carries over to the metaplectic case. Meromorphic continuation follows from computation in Sec.~\ref{sec:computation}.
\end{proof}

Thus we are free to change the order of summation and integration when we evaluate the integrals. By using the results in Sec.~\ref{sec:computation} which is dedicated to evaluating the integrals, we can now prove Theorems.~\ref{thm:distinction-2-eis-pole}, ~\ref{thm:period-res-eis-theta}.

\begin{proof}[Proof of Thm.~\ref{thm:distinction-2-eis-pole}]
 First  assume that $\dim L = 0$. Then the non-distinction conditions for groups `larger than' $G (Z)$ in the assumption
 of the theorem mean that the conditions of  Propositions~\ref{prop:I-Omega10-L-0} and
 \ref{prop:I-Omega11-L-0} are satisfied. Thus the propositions show that $I_{\Omega_{1,0}} (\xi)$ and $I_{\Omega_{1,1}}
 (\xi)$ vanish. 

 Assume that $M (w,s)$ does not have a pole at $s=s_0$. Then by \eqref{eq:I-Omega01-L0-xi-upper-c},
 $I_{\Omega_{0,1}} (\xi_s^c)$ does not have a pole at $s=s_0$. By part (3) of Prop.~\ref{prop:I-Omega01-L-0}, the assumption of
 distinction by $G (Z)$ shows that $I_{\Omega_{0,1}} (\xi_{c,s})$ has a pole at $s=s_0$. This means that $\Lambda^c E
 (g,s,f)$ and as a result $E (g,s,f)$ has a pole at $s=s_0$. We get a contradiction. Thus $M (w,s)$ must have a pole at
 $s=s_0$, which implies that $E (g,s,f)$ has a pole at $s=s_0$.

 Next assume that $\dim L=1$. We note that in Sec.~\ref{sec:case-dim-l-1}, $I_{\Omega_{0,1}} (\xi)$ is further cut into
 3 parts: $J_{\Omega_{0,1},1} (\xi)$, $J_{\Omega_{0,1},2} (\xi)$ and $J_{\Omega_{0,1},3} (\xi)$ and that
 $I_{\Omega_{1,1}} (\xi)$ is further cut into 3 parts: $J_{\Omega_{1,1},1} (\xi)$, $J_{\Omega_{1,1},2} (\xi)$ and
 $J_{\Omega_{1,1},3} (\xi)$.  The non-distinction conditions for groups `larger than' $J(Z,L)$ in the assumption of the
 theorem mean that the conditions of Propositions~\ref{prop:J-Omega01-1-L-1}, \ref{prop:J-Omega01-3-L-1},
 \ref{prop:I-Omega10-L-1}, \ref{prop:J-Omega11-1-L-1}, \ref{prop:J-Omega11-2-L-1}, \ref{prop:J-Omega11-3-L-1} are
 satisfied.  Thus $J_{\Omega_{0,1},1} (\xi)$, $J_{\Omega_{0,1},3} (\xi)$, $I_{\Omega_{1,0}} (\xi)$,
 $J_{\Omega_{1,1},1} (\xi)$, $J_{\Omega_{1,1},2} (\xi)$ and $J_{\Omega_{1,1},3} (\xi)$ all vanish. It remains to
 consider $J_{\Omega_{0,1},2} (\xi)$ which is defined in \eqref{eq:J-Omega01-2}.

 Assume that $M (w,s)$ does not have a pole at $s=s_0$, then by \eqref{eq:J-Omega01-2-L1-xi-upper-c},
 $J_{\Omega_{0,1},2} (\xi_s^c)$ does not have a pole at $s=s_0$. By part (3) of Prop.~\ref{prop:J-Omega01-2-L-1}, the
 assumption of distinction by $J (Z,L)$ shows that $J_{\Omega_{0,1},2} (\xi_{c,s})$ has a pole at $s=s_0$. This means that
 $\Lambda^c E (g,s,f)$ and as a result $E (g,s,f)$ has a pole at $s=s_0$. We get a contradiction. Thus $M (w,s)$ must
 have a pole at $s=s_0$, which implies that $E (g,s,f)$ has a pole at $s=s_0$.
\end{proof}
\begin{proof}[Proof of Thm.~\ref{thm:period-res-eis-theta}]
  Let $s_1>0$. 
  We show that the integral \eqref{eq:period-res-eis-theta} is absolutely convergent and we then check if it vanishes  or not.

  We take residue of the second term in \eqref{eq:truncation} and set
  \begin{align*}
    \theta^c (g) &= \sum_{Q_1 \lmod G (X_1)} \cE_{s_1,Q_1} (\gamma g, f_s)\hat {\tau}^c (H (\gamma g))\\
    &=\sum_{Q_1 \lmod G (X_1)} \res_{s=s_1} (M (w,s)f_s (\gamma g)) \hat {\tau}^c (H (\gamma g)).
  \end{align*}
  Note that this is a finite sum for fixed $g$.
  Then the truncated residue is
  \begin{equation*}
    \Lambda^c \cE_{s_1} (g,f_s) = \cE_{s_1} (g,f_s) - \theta^c (g).
  \end{equation*}
  The integral
  \begin{equation*}
    \int_{[J (Z_1,L)]}  \theta^c (g)\overline {\theta_{X_1,Y} (g,1,\Phi)} dg 
  \end{equation*}
  is absolutely convergent when $\Re (s_1)$ and $c$ are large and has meromorphic continuation to all $s_1\in\CC$ by
  Prop.~\ref{prop:abs-conv}. It is equal to $\res_{s=s_1}I (\xi_s^c)$. Thus
  \begin{align*}
    &\phantom{=}\int_{[J (Z_1,L)]} \cE_{s_1} (g,f_s)\overline {\theta_{X_1,Y} (g,1,\Phi)} dg\\
    &=\int_{[J (Z_1,L)]} (\Lambda^c \cE_{s_1} (g,f_s) + \theta^c (g))\overline {\theta_{X_1,Y} (g,1,\Phi)} dg \\
    &=\res_{s=s_1} (I (\xi_{c,s}) - I (\xi_s^c)) + \res_{s=s_1} I (\xi_s^c)\\
    &=\res_{s=s_1}I (\xi_{c,s}).
  \end{align*}
  In addition, the first equality shows that \eqref{eq:period-res-eis-theta} is absolutely convergent.

  Now if $s_1=s_0$, then by Prop.~\ref{prop:I-Omega01-L-0} for $\dim L=0$ and Prop.~\ref{prop:J-Omega01-2-L-1} for
  $\dim L=1$, \eqref{eq:period-res-eis-theta} is non-vanishing for some choice of data. If $s_1\neq s_0$ or if we
  integrate over $[J (Z''_1,L'')]$ with $(Z'',L'')$ `larger than' $(Z,L)$, again by Prop.~\ref{prop:I-Omega01-L-0} for
  $\dim L=0$ and Prop.~\ref{prop:J-Omega01-2-L-1} for $\dim L=1$, \eqref{eq:period-res-eis-theta} always vanishes.
\end{proof}

\section{Computation of integrals}
\label{sec:computation}
Set $r=\dim X - \dim Z +\dim L$ and $s_0 = (\dim X - (\dim Y +2r ) +2)/2$. The goal of this section is to compute
$I_{\Omega_{0,1}} (\xi)$, $I_{\Omega_{1,0}} (\xi)$ and  $I_{\Omega_{1,1}} (\xi)$ defined in \eqref{eq:I-Omegas}.

\subsection{Case $\dim L =0$}
\label{sec:case-dim-l-0}
We note that since $\dim L=0$, $J (Z_1,L)$ is simply the symplectic group $G (Z_1)$.

\subsubsection{$I_{\Omega_{0,1}} (\xi)$}
\label{sec:Omega01}
We collapse the integral and series. We get
\begin{align*}
 I_{\Omega_{0,1}} (\xi)& = \int_{Q (Z_1,Fe_1^-) (F)\lmod G (Z_1) (\AA)} \xi (g) dg \\
 & = \int_{K_{G (Z_1)}} \int_{[Q (Z_1,Fe_1^-)]} \xi (qk) dq dk\\
  &= \int_{K_{G (Z_1)}} \int_{[G (Z_1)]}\int_{[\GL_1]}\int_{[N (Z_1,Fe_1^-)]} \xi (nm_1 (t)h k)|t|_\AA^{-2\rho_{Q
  (X_1,Fe_1^-)}} dn dt dh dk.
\end{align*}
We note that $\xi$ is the product of two genuine factors, so it is not genuine itself. Let $\xi=\xi_{c,s}$. Then 
\begin{align*}
  \xi (nm_1 (t)h k) &= f_s (nm_1 (t)h k)\overline {\theta_{X_1,Y} (nm_1 (t)h k,1,\Phi)}\hat {\tau}_c (H (nm_1 (t)h k)) \\
  &=\chi_Y\chi_\psi (t)|t|_\AA^{s+\rho_{Q_1}} f_s (hk) \overline {\theta_{X_1,Y} (nm_1 (t)h k,1,\Phi)}\hat {\tau}_c (H
  (m_1 (t))) .
\end{align*}
We integrate over $n\in [N (Z_1,Fe_1^-)]$ first. The only term involving $n$ is $\theta_{X_1,Y}$. If we realise the Weil
representation on the mixed model $\cS ((Y\otimes \ell_1^+ \oplus (Y\otimes X)^+) (\AA))$, then using explicit formulae for the
Weil representation and noting that $Y$ is anisotropic we get
\begin{align*}
  &\phantom{=} \int_{[N (Z_1,Fe_1^-)]}  \theta_{X_1,Y} (nm_1 (t)h k,1,\Phi) dn \\
  &= \int_{[\Hom (\ell_1^+,Z)]} \sum_{w\in (Y\otimes X)^+ (F)} \omega_{X_1,Y} (n (\mu,0)m_1 (t)hk,1) \Phi (0,w) d\mu \\
  &= \sum_{w\in (Y\otimes X)^+ (F)} \omega_{X_1,Y} (m_1 (t)hk,1) \Phi (0,w) \\
  &= \chi_Y\chi_\psi (t)|t|_\AA^{\dim Y /2} \theta_{X,Y} (h,1,\Phi_k).
\end{align*}
Here $\Phi_k (\cdot) = \omega_{X_1,Y} (k,1)\Phi (0,\cdot)$ and for $\mu\in \Hom (\ell_1^+,Z)$, $n (\mu,0)$ denote the
unipotent element in $N (Z_1,Fe_1^-)$ that is characterised by the condition that the $X$-component of the image of
$e_1^+$ under $n (\mu,0)$ is $\mu (e_1^+)$. Next we
integrate over $t\in [\GL_1]$. More precisely, in the following expression, $t$ should be viewed as any element in the pre-image of
$t$ in $\til {\GL}_1 (\AA)$, but the choice does not matter. Excluding terms not involving $t$, we get
\begin{align*}
  &\phantom{=} \int_{[\GL_1]} \chi_Y\chi_\psi (t)|t|_\AA^{s+\rho_{Q_1}}\overline {\chi_Y\chi_\psi (t)|t|_\AA^{\dim Y /2}} \hat {\tau}_c (H
  (m_1 (t))) |t|_\AA^{-2\rho_{Q  (X_1,Fe_1^-)}} dt\\
&  =\vol (F^\times \lmod \AA^1) \int_0^c t^{s-s_0} d^\times t.
\end{align*}
For $\Re s > s_0$, we get
\begin{equation*}
  I_{\Omega_{0,1}} (\xi_{c,s}) = \vol (F^\times \lmod \AA^1) \frac{c^{s-s_0}}{s-s_0} \int_{K_{G (Z_1)}} \int_{[G (Z)]}
  f_s (hk)\overline {\theta_{X,Y} (h,1,\Phi_k)} dh dk.
\end{equation*}
This expression provides the meromorphic continuation of $I_{\Omega_{0,1}} (\xi)$ as a function in $s$.

Next let $\xi=\xi^c_s$. The computation is analogous and we get for $\Re s > -s_0$,
\begin{equation}\label{eq:I-Omega01-L0-xi-upper-c}
   I_{\Omega_{0,1}} (\xi_s^c) = \vol (F^\times \lmod \AA^1) \frac{c^{-s-s_0}}{s+s_0} \int_{K_{G (Z_1)}} \int_{[G (Z)]}
  M (w,s) f_s (hk)\overline {\theta_{X,Y} (h,1,\Phi_k)} dh dk.
\end{equation}

With this computation, we get immediately the first two statements of the next proposition:
\begin{prop}\label{prop:I-Omega01-L-0}
  \begin{enumerate}
  \item If $\sigma\otimes\overline {\Theta_{X,Y}}$ is not $G (Z)$-distinguished, then both $I_{\Omega_{0,1}}
    (\xi_{c,s})$ and $I_{\Omega_{0,1}} (\xi^c_s)$ vanish identically.
  \item $I_{\Omega_{0,1}} (\xi_{c,s})$ does not have a pole at $s\neq s_0$.
  \item If $\sigma\otimes\overline {\Theta_{X,Y}}$ is  $G (Z)$-distinguished, then the residue
    \begin{equation}\label{eq:residue-of-I-Omega01}
      \vol (F^\times \lmod \AA^1)  \int_{K_{G (Z_1)}} \int_{[G (Z)]}
  f_s (hk)\overline {\theta_{X,Y} (h,1,\Phi_k)} dh dk
    \end{equation}
of $I_{\Omega_{0,1}}    (\xi_{c,s})$ at $s=s_0$ does not vanish for some choice of data.
  \end{enumerate}
\end{prop}
\begin{proof}
  The proof of (3) is similar to that in \cite [Prop.~4.1] {MR3805648}. We start with some choice of data such that
  \begin{equation*}
    \int_{[G (Z)]} \phi (h)\overline {\theta_{X,Y} (h,1,\Psi)} dh \neq 0.
  \end{equation*}
Then we can construct a section $f_s\in \cA_1 (s,\chi,\sigma)$ that `extends' $\phi$ and a Schwartz function $\Phi\in\cS
((Y\otimes Fe_1^+ \oplus (Y\otimes X)^+ ) (\AA)) $ that `extends' $\Psi$ such that \eqref{eq:residue-of-I-Omega01} is non-vanishing.
\end{proof}

\subsubsection{$I_{\Omega_{1,0}} (\xi)$}
\label{sec:Omega10}

For each $v\in V-\{0\}$, we look at the $v$-summand of $I_{\Omega_{1,0}} (\xi)$:
\begin{equation}\label{eq:summand-I-Omega10}
  \int_{[G (Z_1)]} \xi (\gamma_{Fv} g) dg.
\end{equation}
Take $v^+$ in $V$ such that $\form {v^+} {v}_V=1$. We take $\gamma_{Fv}$ to be the element in $G (X_1) (F)$ that is
determined by  $e_1^+\leftrightarrow v^+$ and $e_1^-\leftrightarrow v$ on the span of $e_1^+, e_1^-,v^+,v$ and  the identity
action on the orthogonal complement. Then $\gamma_{Fv}G (Z_1)\gamma_{Fv}^{-1} = G (Fv^+\oplus Z \oplus Fv)$. Thus
\eqref{eq:summand-I-Omega10} is equal to
\begin{equation*}
  \int_{[G (Fv^+\oplus Z \oplus Fv)]} \xi ( g\gamma_{Fv}) dg.
\end{equation*}
This is a period integral on $\sigma\otimes\overline {\Theta_{X,Y}}$ over the subgroup $G
(Fv^+\oplus Z \oplus Fv)$ of $G (X)$. Thus we get
\begin{prop}\label{prop:I-Omega10-L-0}
  If $\sigma\otimes\overline {\Theta_{X,Y}}$ is not $G (Z')$-distinguished for all $Z' \supset Z$ with $\dim Z' = \dim Z +2$,
  then $I_{\Omega_{1,0}} (\xi)$ vanishes for $\xi = \xi_{c,s}$ and $\xi_s^c$.
\end{prop}

\subsubsection{$I_{\Omega_{1,1}} (\xi)$}
\label{sec:Omega11}
We collapse the series over $\delta$ and the integral to get
\begin{equation*}
  \int_{J (Z_1,Fe_1^-) (F)\lmod G (Z_1) (\AA)}  \sum_{v\in (V-\{0\})/F^\times} \xi (\gamma_{F (v+e_1^-)} g) dg.
\end{equation*}
Using the Iwasawa decomposition we get that each $v$-term is equal to
\begin{equation*}
 \int_{K_{G (Z_1)}} \int_{\GL_1 (\AA)} \int_{[J (Z_1,Fe_1^-)]} \xi (\gamma_{F (v+e_1^-)} g m_1 (t) k) |t|_\AA^{-2\rho_{Q
     (Z_1,Fe_1^-)}} dg dt dk.
\end{equation*}
We take $\gamma_{F (v+e_1^-)}$ to be the element that is given by $e_1^+\mapsto v^+$, $v^+\mapsto e_1^+ - v^+$,
$v\mapsto e_1^-$ and $e_1^- \mapsto v+e_1^+$ on the span of $e_1^+,e_1^-,v^+, v$ and the identity action on the orthogonal complement.
Since $\gamma_{F (v+e_1^-)} J (Z_1,Fe_1^-) \gamma_{F (v+e_1^-)}^{-1} = J (Fv^+\oplus Z \oplus Fv,Fv)$, we get the inner integral
\begin{equation*}
  \int_{[J (Fv^+\oplus Z \oplus Fv,Fv)]} \xi (gm_1 (t) k)dg
\end{equation*}
which is a period integral on $\sigma\otimes\overline {\Theta_{X,Y}}$ over the subgroup $J (Fv^+\oplus Z \oplus Fv,Fv)$ of $G
(X)$. Thus we get
\begin{prop}\label{prop:I-Omega11-L-0}
  If $\sigma\otimes\overline {\Theta_{X,Y}}$ is not $J (Z',L')$-distinguished for all $Z' \supset Z$ with $\dim Z' = \dim Z +2$
  and $L'$ an isotropic line of $Z'$ in the orthogonal complement of $Z$,
  then $I_{\Omega_{1,1}} (\xi)$ vanishes for $\xi = \xi_{c,s}$ and $\xi_s^c$.
\end{prop}

\subsection{Case $\dim L = 1$}
\label{sec:case-dim-l-1}

\subsubsection{$I_{\Omega_{0,1}} (\xi)$}
\label{sec:omega01-L1}

We consider the $J (Z_1,L)$-orbits in $Q (Z_1,Fe_1^-)\lmod G (Z_1)$ or equivalently in the set of isotropic lines in
$Z_1$. For an isotropic line $\ell$ in $Z_1$, let $\delta_\ell$ denote any element in $G (Z_1)$ such that $Fe_1^-
\delta_\ell = \ell$. There are three $J (Z_1,L)$-orbits. Let $f_1^-$ be any non-zero element in $L$ and $f_1^+$ be an element in $Z$ such that $\form {f_1^+} {f_1^-}_Z=1$. Write $Z=Ff_1^+ \oplus W \oplus Ff_1^-$ and form the augmented space $W_1=\ell_1^+\oplus W \oplus \ell^-$. We note that $Z_1=Ff_1^+ \oplus W_1 \oplus Ff_1^-$. An isotropic line $F (af_1^+ + w + bf_1^-)$ in $Z_1$ for $a,b\in F$ and
$w\in W_1$ is in the same $J (Z_1,L)$-orbit as
\begin{align*}
  &Ff_1^+,\quad \text {if $a\neq 0$};\\
  &Fe_1^-,\quad \text {if $a=0$ and $w\neq 0$};\\
  &Ff_1^-, \quad\text {if $a= 0$ and $w=0$}.
\end{align*}
The stabiliser of $Ff_1^+$ in $J (Z_1,L)$ is $G (W_1)$ and we have $\Stab_{J (Z_1,L)} Ff_1^+ \lmod J (Z_1,L) \isom N
(Z_1,L)$. To describe the stabiliser of $Fe_1^-$, let $NJ_1$ be the subgroup of $J (Z_1,L)$ consisting of
elements of the form
\begin{equation*}
  \begin{pmatrix} 1 & * &0 &0&0\\ & 1 &0&0&0\\ &&I &0&0\\ &&&1& *\\
&&&&1
  \end{pmatrix}
\end{equation*} with respect to the `basis' $f_1^+, e_1^+, W, e_1^-,
f_1^-$ and $NJ_2$  the subgroup of $J (Z_1,L)$ consisting of
elements of the form
\begin{equation*}
  \begin{pmatrix} 1 & 0 &* &*&*\\ & 1 &0&0&*\\ &&I &0&*\\ &&&1&0 \\
&&&&1
  \end{pmatrix}.
\end{equation*}
Then the stabiliser of $Fe_1^-$ in $J (Z_1,L)$ is $NJ_2 \rtimes Q (W_1,Fe_1^-)$. For $Ff_1^-$ the stabiliser is the full
$J (Z_1,L)$. Thus $I_{\Omega_{0,1}} (\xi)$ is further split into three parts:
\begin{align}
  \label{eq:J-Omega01-1}
  J_{\Omega_{0,1},1} (\xi) &= \int_{[J (Z_1,L)]} \sum_{\eta\in N (Z_1,L)} \xi (\delta_{Ff_1^+}\eta g ) dg;\\
\label{eq:J-Omega01-2}
  J_{\Omega_{0,1},2} (\xi) &= \int_{[J (Z_1,L)]} \sum_{\eta\in NJ_2 \rtimes Q (W_1,Fe_1^-) \lmod J (Z_1,L)} \xi (\eta g )
  dg;\\
  \label{eq:J-Omega01-3}
  J_{\Omega_{0,1},3} (\xi) &= \int_{[J (Z_1,L)]} \xi (\delta_{Ff_1^-} g ) dg.
\end{align}

Now we evaluate them one by one.

We collapse the series and the integral and find that \eqref{eq:J-Omega01-1} is equal to
\begin{equation*}
  \int_{N (Z_1,L) (\AA)}  \int_{[G (W_1)]}\xi (\delta_{Ff_1^+} gn ) dg dn.
\end{equation*}
Consider the inner integral. We may pick $\delta_{Ff_1^+}$ to be the element that is given by $e_1^+ \leftrightarrow
-f_1^-$ and $e_1^- \leftrightarrow f_1^+$ on the span of $e_1^\pm, f_1^{\pm}$ and identity on the orthogonal
complement. Since $\delta_{Ff_1^+} G (W_1) \delta_{Ff_1^+}^{-1} = G (Ff_1^+ \oplus W \oplus Ff_1^-) = G (Z)$, 
this inner integral is a period integral on $\sigma\otimes\overline {\Theta_{X,Y}}$ over the subgroup $G (Z)$ of $G
(X)$. Thus we get
\begin{prop}\label{prop:J-Omega01-1-L-1}
  If $\sigma\otimes\overline {\Theta_{X,Y}}$ is not $G (Z)$-distinguished, then $J_{\Omega_{0,1},1} (\xi)$ vanishes for
  $\xi=\xi_{c,s}$ and $\xi_s^c$.
\end{prop}

We collapse the series and the integral and find that \eqref{eq:J-Omega01-2} is equal to
\begin{equation*}
  \int_{NJ_1 (\AA)}\int_{[NJ_2]}\int_{Q (W_1,Fe_1^-) (F)\lmod G (W_1) (\AA)} \xi (n_2 n_1 g) dg dn_2 dn_1.
\end{equation*}
Then using the Iwasawa decomposition of $G (W_1) (\AA)$ we get
\begin{align*}
&\phantom{=} \int_{K_{G (W_1)}}\int_{NJ_1 (\AA)}\int_{[NJ_2]}\int_{[N (W_1,Fe_1^-)]}\int_{[G (W)]}\int_{[GL_1]} \xi (n_2 n_1 n m_1
  (t) g k)\\
 & \qquad\qquad|t|_\AA^{-2\rho_{Q ((W_1,Fe_1^-))}} dt dg dn dn_2 dn_1 dk\\
&  =   \int_{K_{G (W_1)}}\int_{NJ_1 (\AA)}\int_{[NJ_2]}\int_{[N (W_1,Fe_1^-)]}\int_{[G (W)]}\int_{[GL_1]} \xi (n n_2 m_1
                                                                 (t)  g n_1 k)\\
  &\qquad\qquad|t|_\AA^{-2\rho_{Q ((W_1,Fe_1^-))} - 1} dt dg dn dn_2 dn_1 dk.
\end{align*}
Now we collapse part of $NJ_2$ and $G (W)$ to get $J (Z,Ff_1^-)$. We get
\begin{align*}
  \int_{K_{G (W_1)}}\int_{NJ_1 (\AA)}\int_{[N']}\int_{[J (Z,f_1^-)]}\int_{[GL_1]} \xi (n' m_1
  (t)  g n_1 k) |t|_\AA^{-2\rho_{Q ((W_1,Fe_1^-))} - 1} dt dg dn' dk.
\end{align*}
Note that $N'=N (Z_1,Fe_1^-)\cap N (Z_1,Fe_1^-\oplus Ff_1^-)$ is formed out of the leftover part of $NJ_2$ and $N (W_1,Fe_1^-)$.

Suppose that $\xi=\xi_{c,s}$. Then
\begin{align*}
 &\phantom{=}  \xi (n' m_1  (t)  g n_1 k) \\
  &= f_s (n' m_1  (t)  g n_1 k)\overline {\theta_{X_1,Y} (n' m_1  (t)  g n_1 k,1,\Phi)}\hat
  {\tau}_c (H (n' m_1  (t)  g n_1 k)) \\
&  =\chi_Y\chi_\psi (m_1 (t)) |t|_\AA^{s+\rho_{Q_1}} f_s (  g n_1 k)\overline {\theta_{X_1,Y} (n' m_1  (t)  g n_1 k,1,\Phi)}\hat
  {\tau}_c (H ( m_1  (t)   n_1 )).
\end{align*}
The only term that involves $n'$ is $\overline {\theta_{X_1,Y} (n' m_1  (t)  g n_1 k,1,\Phi)}$. By the explicit formulae of the
Weil representation, we get
\begin{equation*}
  \int_{[N']}  \theta_{X_1,Y} (n' m_1  (t)  g n_1 k,1,\Phi) dn' = \chi_Y\chi_\psi (t)|t|_\AA^{\dim Y /2} \theta_{X,Y}
  (g,1,\Phi_{n_1 k}),
\end{equation*}
where $\Phi_{n_1 k} (\cdot) = \omega_{X_1,Y} (n_1k,1)\Phi (0,\cdot)$. In total, the exponent of $|t|_\AA$ is
\begin{equation*}
  s+\rho_{Q_1} +\dim Y/2 - 2\rho_{Q ((W_1,Fe_1^-))} - 1 =s-s_0.
\end{equation*}
Thus $J_{\Omega_{0,1},1} (\xi_{c,s})$ is equal to
\begin{align*}
  &\phantom{=} \int_{K_{G (W_1)}}\int_{NJ_1 (\AA)} \int_{[J (Z,L)]} f_s (g\overline {n} k) \overline {\theta_{X,Y} (g,1,\Phi_{\overline {n}k})} \int_{[GL_1]} \hat
  {\tau}_c (H ( m_1  (t)   n_1 )) |t|_\AA^{s-s_0} dt dg d\overline {n} dk\\
&  = \vol (F^\times \lmod \AA^\times)  \int_{K_{G (W_1)}}\int_{NJ_1 (\AA)} \int_{[J (Z,L)]} f_s (g\overline {n} k) \overline {\theta_{X,Y} (g,1,\Phi_{\overline {n}k})} \frac{(c/c (\overline {n}))^{s-s_0}}{s-s_0} dg d\overline {n} dk,
\end{align*}
where $c (\overline {n}) = \exp (H (\overline {n}))$ where we note that $H (\overline {n})\in \fa_{M_1} \isom \RR $.

Thus it can possibly have a pole only at $s=s_0$ with residue
\begin{equation*}
  \vol (F^\times \lmod \AA^\times)  \int_{K_{G (W_1)}}\int_{NJ_1 (\AA)} \int_{[J (Z,L)]} f_{s_0} (g\overline {n} k) \overline {\theta_{X,Y} (g,1,\Phi_{\overline {n}k})}  dg d\overline {n} dk.
\end{equation*}
Note that the innermost integral is a period integral on $\sigma\otimes\overline {\Theta_{X,Y}}$ over the subgroup $J
(Z,L)$ of $G(X)$.
Similarly we evaluate $J_{\Omega_{0,1},1} (\xi_s^c)$ to get
\begin{align}\label{eq:J-Omega01-2-L1-xi-upper-c}
&  \vol (F^\times \lmod \AA^\times)  \int_{K_{G (W_1)}}\int_{NJ_1 (\AA)} \int_{[J (Z,L)]} M (w,s) f_s (g\overline {n} k)
  \overline {\theta_{X,Y} (g,1,\Phi_{\overline {n}k})}\\
 \nonumber &\qquad\qquad\frac{(c/c (\overline {n}))^{-s-s_0}}{s+s_0} dg d\overline {n} dk.
\end{align}
Thus we get parts (1) and (2) of the following
\begin{prop}\label{prop:J-Omega01-2-L-1}
  \begin{enumerate}
  \item If $\sigma\otimes\overline {\Theta_{X,Y}}$ is not $J (Z,L)$-distinguished, then $J_{\Omega_{0,1},2} (\xi)$
    vanishes identically for $\xi=\xi_{c,s}$ and $\xi_s^c$;
  \item $J_{\Omega_{0,1},2} (\xi)$ does not have a pole at $s\neq s_0$;
  \item If $\sigma\otimes\overline {\Theta_{X,Y}}$ is $J (Z,L)$-distinguished, then the residue
    \begin{equation}\label{eq:residue-of-J-Omega01-2}
      \vol (F^\times \lmod \AA^\times)  \int_{K_{G (W_1)}}\int_{NJ_1 (\AA)} \int_{[J (Z,L)]} f_{s_0} (g\overline {n} k) \overline {\theta_{X,Y} (g,1,\Phi_{\overline {n}k})}  dg d\overline {n} dk
    \end{equation}
    of $J_{\Omega_{0,1},2} (\xi_{c,s})$ does not vanish for some choice of data. 
  \end{enumerate}
\end{prop}
\begin{proof}
    For the proof of (3), it can be checked that we can extend the proof in \cite [Prop.~4.5] {MR3805648} for the symplectic case to the
    metaplectic case. The proof there is quite technical due to the presence of the integration over $[NJ_1]$ which is a
    unipotent subgroup.  We start with some choice of data such that
  \begin{equation*}
    \int_{[J (Z,L)]} \phi (h)\overline {\theta_{X,Y} (h,1,\Psi)} dh \neq 0.
  \end{equation*}
Then we can construct a section $f_s\in \cA_1 (s,\chi,\sigma)$ that `extends' $\phi$ and a Schwartz function $\Phi\in\cS
(Y\otimes Fe_1^+ (\AA)) \hat {\otimes} \cS_{X,Y} $ that `extends' $\Psi$ such that \eqref{eq:residue-of-J-Omega01-2} is non-vanishing.
\end{proof}

Finally we evaluate \eqref{eq:J-Omega01-3}. We take $\delta_{Ff_1^-}$ to be the element
that is given by $e_1^+ \leftrightarrow f_1^+$ and $e_1^- \leftrightarrow f_1^-$  on the span of
$e_1^\pm, f_1^\pm$ and identity on the orthogonal complement. Then $\delta_{Ff_1^-}J
(Z_1,L)\delta_{Ff_1^-}^{-1} = J (Z_1,Fe_1^-)$. Thus \eqref{eq:J-Omega01-3} is equal to
\begin{equation*}
  \int_{[J (Z_1,Fe_1^-)]} \xi (g\delta_{Ff_1^-}) dg = \int_{[N (Z_1,Fe_1^-)]}\int_{[G (Z)]} \xi (gn\delta_{Ff_1^-}) dgdn
\end{equation*}
whose inner integral  is a period integral on $\sigma\otimes\overline {\Theta_{X,Y}}$ over $[G
(Z)]$. Thus we get:
\begin{prop}\label{prop:J-Omega01-3-L-1}
  If $\sigma\otimes\overline {\Theta_{X,Y}}$ is not $G (Z)$-distinguished, then $J_{\Omega_{0,1},3}
  (\xi)$ vanishes for $\xi_{c,s}$ and $\xi_s^c$.
\end{prop}

\subsubsection{$I_{\Omega_{1,0}} (\xi)$}
\label{sec:Omega10-L1}

For each $v\in V-\{0\}$, fix $v^+\in V$ that is dual to $v$, i.e., $\form{v^+}{v}=1$. Then we take $\gamma_{Fv}$ that is
determined by $v^+\leftrightarrow e_1^+$ and $v \leftrightarrow e_1^-$ on the span of $e_1^+,e_1^-, v^+,v$ and identity
on the orthogonal complement. Since $\gamma_{Fv} J (Z_1,L) \gamma_{Fv}^{-1} \isom J (Fv^+\oplus Z \oplus Fv, L)$,  $I_{\Omega_{1,0}} (\xi)$ is equal to
\begin{equation*}
  \sum_{v\in (V-\{0\} / F^\times)} \int_{[J (Fv^+\oplus Z \oplus Fv, L)]} \xi (g\gamma_{Fv}) dg
\end{equation*}
which is a sum of period integrals on $\sigma\otimes\overline {\Theta_{X,Y}}$ over $[J (Fv^+\oplus Z \oplus Fv,
L)]$. Thus we get
\begin{prop}\label{prop:I-Omega10-L-1}
  If $\sigma\otimes\overline {\Theta_{X,Y}}$ is not $J (Z',L)$-distinguished for all $Z'\supset Z$ such that $\dim Z'
  =\dim Z +2$, then $I_{\Omega_{1,0}} (\xi)$ vanishes for $\xi=\xi_{c,s}$ and $\xi_s^c$.
\end{prop}

\subsubsection{$I_{\Omega_{1,1}} (\xi)$}
\label{sec:Omega11-L1}

We fix $v\in V-\{0\}$ and consider each $v$-term
\begin{align*}
  \int_{[J (Z_1,L)]} \sum_{\delta \in J(Z_1,Fe_1^-) \lmod G (Z_1)}  \xi (\gamma_{F (v+e_1^-)}\delta g) dg.
\end{align*}
We note that $J(Z_1,Fe_1^-) \lmod G (Z_1)$ parametrises the set of non-zero isotropic vectors in $Z_1$. We consider the
$J (Z_1,L)$-orbits on this set. The analysis is similar to that in Sec.~\ref{sec:omega01-L1} where we considered isotropic
lines in $Z_1$. Again we have $L=Ff_1^-$ and   an element $f_1^+$ in $Z$ such that $\form {f_1^+} {f_1^-}_Z=1$. We write $Z=Ff_1^+ \oplus W \oplus Ff_1^-$ and form the augmented space $W_1$. We note that $Z_1=Ff_1^+ \oplus W_1 \oplus Ff_1^-$. A non-zero isotropic vector $ af_1^+ + w + bf_1^-$ in $Z_1$ for $a,b\in F$ and
$w\in W_1$ is in the same $J (Z_1,L)$-orbit as
\begin{align*}
&  af_1^+,\quad \text {if $a\neq 0$};\\
 & e_1^-, \quad\text {if $a=0$ and $w\neq 0$};\\
&  bf_1^-, \quad\text {if $a= 0$ and $w=0$}.
\end{align*}
For $a\in F^\times$, the stabiliser of $af_1^+$ in $J (Z_1,L)$ is $G (W_1)$ and we have 
\begin{equation*}
  \Stab_{J (Z_1,L)} af_1^+ \lmod  J (Z_1,L) \isom N(Z_1,L).
\end{equation*}
 The stabiliser of $e_1^-$ in $J (Z_1,L)$ is $NJ_2 \rtimes J (W_1,Fe_1^-)$, where we
adopt the same notation $NJ_1$ and $NJ_2$ as in Sec.~\ref{sec:omega01-L1}. For $b\in F^\times$, the stabiliser of $bf_1^-$ in
$J (Z_1,L)$ is $J (Z_1,L)$. Thus $I_{\Omega_{1,1}} (\xi)$ is further split into three parts:
\begin{align}
  \label{eq:J-Omega11-1}
  J_{\Omega_{1,1},1} (\xi)& = \int_{[J (Z_1,L)]} \sum_{a\in F^\times} \sum_{\eta\in N (Z_1,L)} \xi (\gamma_{F (v+e_1^-)}\delta_{af_1^+}\eta g ) dg;\\
\label{eq:J-Omega11-2}
  J_{\Omega_{1,1},2} (\xi) &= \int_{[J (Z_1,L)]} \sum_{\eta\in NJ_2 \rtimes J (W_1,Fe_1^-) \lmod J (Z_1,L)} \xi (\gamma_{F (v+e_1^-)}\eta g )
  dg;\\
  \label{eq:J-Omega11-3}
  J_{\Omega_{1,1},3} (\xi) &= \int_{[J (Z_1,L)]} \sum_{b\in F^\times} \xi (\gamma_{F (v+e_1^-)}\delta_{bf_1^-} g ) dg.
\end{align}

We take the same $\gamma_{F (v+e_1^-)}$ as in Sec.~\ref{sec:Omega11}. We proceed to evaluate each part.

For each $a$-term in $J_{\Omega_{1,1},1} (\xi)$, we get
\begin{equation*}
  \int_{N (Z_1,L) (\AA)}  \int_{[G (W_1)]}\xi (\gamma_{F (v+e_1^-)}\delta_{af_1^+} g n) dg dn.
\end{equation*}
Since $\delta_{af_1^+} G (W_1) \delta_{af_1^+}^{-1} =  G (Z)$ and $\gamma_{F (v+e_1^-)}$ commutes with $G (Z)$, the
inner integral becomes
\begin{equation*}
  \int_{[G (Z)]}\xi (g \gamma_{F (v+e_1^-)}\delta_{af_1^+}  n) dg
\end{equation*}
which is a period integral on $\sigma\otimes\overline {\Theta_{X,Y}}$ over the subgroup $G (Z)$ of $G
(X)$. We get
\begin{prop}\label{prop:J-Omega11-1-L-1}
  If $\sigma\otimes\overline {\Theta_{X,Y}}$ is not $G (Z)$-distinguished, then $J_{\Omega_{1,1},1} (\xi)$ vanishes for
  $\xi=\xi_{c,s}$ and $\xi_s^c$.
\end{prop}

For $J_{\Omega_{1,1},2} (\xi)$ we collapse the integral and the series to get
\begin{align*}
  \int_{NJ_1 (\AA)}\int_{[NJ_2]}\int_{J (W_1,Fe_1^-) (F)\lmod G (W_1) (\AA)}   \xi (\gamma_{F (v+e_1^-)} n_1 n_2 g )  dg
  dn_2 dn_1.
\end{align*}
Then using the Iwasawa decomposition for $G (W_1) (\AA)$ we get
\begin{align*}
&\phantom{=}   \int_{K_{G (W_1)}}\int_{NJ_1 (\AA)}\int_{[NJ_2]}\int_{[J (W_1,Fe_1^-)]} \int_{GL_1 (\AA)}
  \xi (\gamma_{F (v+e_1^-)} n_1 n_2   g m_1 (t) k)  dt dg dn_2 dn_1 dk\\
&  =   \int_{K_{G (W_1)}}\int_{NJ_1 (\AA)}  \int_{GL_1 (\AA)} \int_{[J (Z_1,Fe_1^-\oplus Ff_1^-)]}
  \xi (\gamma_{F (v+e_1^-)} n_1   g m_1 (t) k)   dg dt dn_1 dk \\
 &  =   \int_{K_{G (W_1)}}\int_{NJ_1 (\AA)}  \int_{GL_1 (\AA)} \int_{[J (Z_1,Fe_1^-\oplus Ff_1^-)]}
  \xi (\gamma_{F (v+e_1^-)}   g n_1  m_1 (t) k)   dg dt dn_1 dk .
\end{align*}
Since $\gamma_{F (v+e_1^-)} J (Z_1,Fe_1^-\oplus Ff_1^-)\gamma_{F (v+e_1^-)}^{-1} = J (Fv^+ \oplus Z \oplus Fv, Ff_1^-
\oplus Fv)$, the innermost integral is equal to
\begin{equation*}
  \int_{[J (Fv^+ \oplus Z \oplus Fv, Ff_1^-\oplus Fv)]}  \xi (g \gamma_{F (v+e_1^-)}  n_1  m_1 (t) k)   dg 
\end{equation*}
which is a period integral on $\sigma\otimes\overline {\Theta_{X,Y}}$ over $J (Fv^+ \oplus Z \oplus Fv, Ff_1^-
\oplus Fv)$. Thus we get
\begin{prop}\label{prop:J-Omega11-2-L-1}
  If $\sigma\otimes\overline {\Theta_{X,Y}}$ is not $J (Z',L'\oplus L)$-distinguished for any $Z'\supset Z$ such that
  $\dim Z' = \dim Z +2$ and $L'$  an isotropic line of $Z'$ in the orthogonal complement of $Z$, then
  $J_{\Omega_{1,1},2} (\xi)$ vanishes for
  $\xi=\xi_{c,s}$ and $\xi_s^c$.
\end{prop}

For each $b$-term in $J_{\Omega_{1,1},3} (\xi)$, we take $\delta_{bf_1^-}$ to be the element that is given by $e_1^+
\leftrightarrow f_1^+$ and $e_1^-
\leftrightarrow f_1^-$ and identity on the orthogonal complement. We evaluate
\begin{equation*}
  \int_{[J (Z_1,L)]}  \xi (\gamma_{F (v+e_1^-)}\delta_{bf_1^-} g ) dg.
\end{equation*}
We have $\delta_{bf_1^-}J (Z_1,L) \delta_{bf_1^-}^{-1} = J (Z_1,Fe_1^-)$ and $\gamma_{F (v+e_1^-)} J
(Z_1,Fe_1^-) \gamma_{F (v+e_1^-)}^{-1} = J (Fv^+ \oplus Z \oplus Fv, Fv)$. Thus we get
\begin{equation*}
  \int_{[J (Fv^+ \oplus Z \oplus Fv, Fv)]}  \xi ( g\gamma_{F (v+e_1^-)}\delta_{bf_1^-} ) dg
\end{equation*}
which is a period integral on $\sigma\otimes\overline {\Theta_{X,Y}}$ over $J (Fv^+ \oplus Z \oplus Fv, Fv)$. We get
\begin{prop}\label{prop:J-Omega11-3-L-1}
  If $\sigma\otimes\overline {\Theta_{X,Y}}$ is not $J (Z',L')$-distinguished for any $Z'\supset Z$ such that
  $\dim Z' = \dim Z +2$ and $L'$ an isotropic line of $Z'$ in the orthogonal complement of $Z$, then
  $J_{\Omega_{1,1},3} (\xi)$ vanishes for
  $\xi=\xi_{c,s}$ and $\xi_s^c$.
\end{prop}


\begin{thebibliography}{BMM17}

\bibitem[Art78]{MR518111}
James~G. Arthur.
\newblock A trace formula for reductive groups. {I}. {T}erms associated to
  classes in {$G({\bf Q})$}.
\newblock {\em Duke Math. J.}, 45(4):911--952, 1978.

\bibitem[Art80]{MR558260}
James Arthur.
\newblock A trace formula for reductive groups. {II}. {A}pplications of a
  truncation operator.
\newblock {\em Compositio Math.}, 40(1):87--121, 1980.

\bibitem[Art13]{MR3135650}
James Arthur.
\newblock {\em The endoscopic classification of representations}, volume~61 of
  {\em American Mathematical Society Colloquium Publications}.
\newblock American Mathematical Society, Providence, RI, 2013.
\newblock Orthogonal and symplectic groups.

\bibitem[BMM16]{MR3508219}
Nicolas Bergeron, John Millson, and Colette Moeglin.
\newblock The {H}odge conjecture and arithmetic quotients of complex balls.
\newblock {\em Acta Math.}, 216(1):1--125, 2016.

\bibitem[BMM17]{MR3821921}
Nicolas Bergeron, John Millson, and Colette Moeglin.
\newblock Hodge type theorems for arithmetic hyperbolic manifolds.
\newblock In {\em Geometry, analysis and probability}, volume 310 of {\em
  Progr. Math.}, pages 47--56. Birkh\"{a}user/Springer, Cham, 2017.

\bibitem[Gan12]{MR3006697}
Wee~Teck Gan.
\newblock Doubling zeta integrals and local factors for metaplectic groups.
\newblock {\em Nagoya Math. J.}, 208:67--95, 2012.

\bibitem[Gao18]{MR3749191}
Fan Gao.
\newblock The {L}anglands-{S}hahidi {$L$}-functions for {B}rylinski-{D}eligne
  extensions.
\newblock {\em Amer. J. Math.}, 140(1):83--137, 2018.

\bibitem[GI14]{MR3166215}
Wee~Teck Gan and Atsushi Ichino.
\newblock Formal degrees and local theta correspondence.
\newblock {\em Invent. Math.}, 195(3):509--672, 2014.

\bibitem[GI18]{MR3866889}
Wee~Teck Gan and Atsushi Ichino.
\newblock The {S}himura-{W}aldspurger correspondence for {${\rm Mp}_{2n}$}.
\newblock {\em Ann. of Math. (2)}, 188(3):965--1016, 2018.

\bibitem[GJS09]{MR2540878}
David Ginzburg, Dihua Jiang, and David Soudry.
\newblock Poles of {$L$}-functions and theta liftings for orthogonal groups.
\newblock {\em J. Inst. Math. Jussieu}, 8(4):693--741, 2009.

\bibitem[GQT14]{MR3279536}
Wee~Teck Gan, Yannan Qiu, and Shuichiro Takeda.
\newblock The regularized {S}iegel-{W}eil formula (the second term identity)
  and the {R}allis inner product formula.
\newblock {\em Invent. Math.}, 198(3):739--831, 2014.

\bibitem[GS]{gan:_repres_metap_group_i}
Wee~Teck Gan and Gordan Savin.
\newblock Representations of metaplectic groups i: Epsilon dichotomy and local
  langlands correspondence.
\newblock preprint.

\bibitem[Hen00]{MR1738446}
Guy Henniart.
\newblock Une preuve simple des conjectures de {L}anglands pour {${\rm GL}(n)$}
  sur un corps {$p$}-adique.
\newblock {\em Invent. Math.}, 139(2):439--455, 2000.

\bibitem[HH12]{MR3012154}
Hongyu He and Jerome~William Hoffman.
\newblock Picard groups of {S}iegel modular 3-folds and {$\theta$}-liftings.
\newblock {\em J. Lie Theory}, 22(3):769--801, 2012.

\bibitem[HT01]{MR1876802}
Michael Harris and Richard Taylor.
\newblock {\em The geometry and cohomology of some simple {S}himura varieties},
  volume 151 of {\em Annals of Mathematics Studies}.
\newblock Princeton University Press, Princeton, NJ, 2001.
\newblock With an appendix by Vladimir G. Berkovich.

\bibitem[Ich01]{MR1863861}
Atsushi Ichino.
\newblock On the regularized {S}iegel-{W}eil formula.
\newblock {\em J. Reine Angew. Math.}, 539:201--234, 2001.

\bibitem[Ike92]{MR1174424}
Tamotsu Ikeda.
\newblock On the location of poles of the triple {$L$}-functions.
\newblock {\em Compositio Math.}, 83(2):187--237, 1992.

\bibitem[Ike96]{MR1411571}
Tamotsu Ikeda.
\newblock On the residue of the {E}isenstein series and the {S}iegel-{W}eil
  formula.
\newblock {\em Compositio Math.}, 103(2):183--218, 1996.

\bibitem[Jia14]{jiang14:_autom_integ_trans_class_group_i}
Dihua Jiang.
\newblock Automorphic integral transforms for classical groups i: Endoscopy
  correspondences.
\newblock In James~W. Cogdell, Freydoon Shahidi, and David Soudry, editors,
  {\em Automorphic Forms and Related Geometry: Assessing the Legacy of I.I.
  Piatetski-Shapiro}, volume 614 of {\em Contemporary Mathematics}. American
  Mathematical Society, Providence, RI, 2014.

\bibitem[JL18]{MR3969876}
Dihua Jiang and Baiying Liu.
\newblock Fourier coefficients and cuspidal spectrum for symplectic groups.
\newblock In {\em Geometric aspects of the trace formula}, Simons Symp., pages
  211--244. Springer, Cham, 2018.

\bibitem[JS07]{MR2330445}
Dihua Jiang and David Soudry.
\newblock On the genericity of cuspidal automorphic forms of {${\rm
  SO}(2n+1)$}. {II}.
\newblock {\em Compos. Math.}, 143(3):721--748, 2007.

\bibitem[JW16]{MR3435720}
Dihua Jiang and Chenyan Wu.
\newblock On {$(\chi,b)$}-factors of cuspidal automorphic representations of
  unitary groups {I}.
\newblock {\em J. Number Theory}, 161:88--118, 2016.

\bibitem[JW18]{MR3805648}
Dihua Jiang and Chenyan Wu.
\newblock Periods and ({$\chi$}, b)-factors of cuspidal automorphic forms of
  symplectic groups.
\newblock {\em Israel J. Math.}, 225(1):267--320, 2018.

\bibitem[KMSW]{kaletha:_endos_class_repres}
Tasho Kaletha, Alberto Minguez, Sug~Woo Shin, and Paul-James White.
\newblock Endoscopic classification of representations: Inner forms of unitary
  groups.
\newblock arXiv:1409.3731.

\bibitem[KR94]{MR1289491}
Stephen~S. Kudla and Stephen Rallis.
\newblock A regularized {S}iegel-{W}eil formula: the first term identity.
\newblock {\em Ann. of Math. (2)}, 140(1):1--80, 1994.

\bibitem[Kud94]{MR1286835}
Stephen~S. Kudla.
\newblock Splitting metaplectic covers of dual reductive pairs.
\newblock {\em Israel J. Math.}, 87(1-3):361--401, 1994.

\bibitem[Lan71]{MR0419366}
Robert~P. Langlands.
\newblock {\em Euler products}.
\newblock Yale University Press, New Haven, Conn.-London, 1971.
\newblock A James K. Whittemore Lecture in Mathematics given at Yale
  University, 1967, Yale Mathematical Monographs, 1.

\bibitem[Lan89]{MR1011897}
R.~P. Langlands.
\newblock On the classification of irreducible representations of real
  algebraic groups.
\newblock In {\em Representation theory and harmonic analysis on semisimple
  {L}ie groups}, volume~31 of {\em Math. Surveys Monogr.}, pages 101--170.
  Amer. Math. Soc., Providence, RI, 1989.

\bibitem[LR05]{MR2192828}
Erez~M. Lapid and Stephen Rallis.
\newblock On the local factors of representations of classical groups.
\newblock In {\em Automorphic representations, {$L$}-functions and
  applications: progress and prospects}, volume~11 of {\em Ohio State Univ.
  Math. Res. Inst. Publ.}, pages 309--359. de Gruyter, Berlin, 2005.

\bibitem[M{\oe}g97]{MR1473165}
Colette M{\oe}glin.
\newblock Non nullit\'e de certains rel\^evements par s\'eries th\'eta.
\newblock {\em J. Lie Theory}, 7(2):201--229, 1997.

\bibitem[Mok15]{MR3338302}
Chung~Pang Mok.
\newblock Endoscopic classification of representations of quasi-split unitary
  groups.
\newblock {\em Mem. Amer. Math. Soc.}, 235(1108):vi+248, 2015.

\bibitem[MW95]{MR1361168}
C.~M{\oe}glin and J.-L. Waldspurger.
\newblock {\em Spectral decomposition and {E}isenstein series}, volume 113 of
  {\em Cambridge Tracts in Mathematics}.
\newblock Cambridge University Press, Cambridge, 1995.
\newblock Une paraphrase de l'{\'E}criture [A paraphrase of Scripture].

\bibitem[MW16a]{MR3823813}
Colette Moeglin and Jean-Loup Waldspurger.
\newblock {\em Stabilisation de la formule des traces tordue. {V}ol. 1}, volume
  316 of {\em Progress in Mathematics}.
\newblock Birkh\"{a}user/Springer, Cham, 2016.

\bibitem[MW16b]{MR3823814}
Colette Moeglin and Jean-Loup Waldspurger.
\newblock {\em Stabilisation de la formule des traces tordue. {V}ol. 2}, volume
  317 of {\em Progress in Mathematics}.
\newblock Birkh\"{a}user/Springer, Cham, 2016.

\bibitem[PSR87]{ps-rallis87_l_funct_for_the_class_group}
Ilya Piatetski-Shapiro and Stephen Rallis.
\newblock {$L$}-functions for the classical groups.
\newblock In {\em Explicit constructions of automorphic {$L$}-functions},
  volume 1254 of {\em Lecture Notes in Mathematics}, pages 1--52.
  Springer-Verlag, Berlin, 1987.

\bibitem[Ral82]{MR658543}
Stephen Rallis.
\newblock Langlands' functoriality and the {W}eil representation.
\newblock {\em Amer. J. Math.}, 104(3):469--515, 1982.

\bibitem[RR93]{MR1197062}
R.~Ranga~Rao.
\newblock On some explicit formulas in the theory of {W}eil representation.
\newblock {\em Pacific J. Math.}, 157(2):335--371, 1993.

\bibitem[Sar05]{MR2192019}
Peter Sarnak.
\newblock Notes on the generalized {R}amanujan conjectures.
\newblock In {\em Harmonic analysis, the trace formula, and {S}himura
  varieties}, volume~4 of {\em Clay Math. Proc.}, pages 659--685. Amer. Math.
  Soc., Providence, RI, 2005.

\bibitem[Sch13]{MR3049932}
Peter Scholze.
\newblock The local {L}anglands correspondence for {$\rm{GL}_n$} over
  {$p$}-adic fields.
\newblock {\em Invent. Math.}, 192(3):663--715, 2013.

\bibitem[Sha10]{MR2683009}
Freydoon Shahidi.
\newblock {\em Eisenstein series and automorphic {$L$}-functions}, volume~58 of
  {\em American Mathematical Society Colloquium Publications}.
\newblock American Mathematical Society, Providence, RI, 2010.

\bibitem[Wu17]{MR3595433}
ChenYan Wu.
\newblock A critical case of {R}allis inner product formula.
\newblock {\em Sci. China Math.}, 60(2):201--222, 2017.

\bibitem[Wu21]{wu21:chi-b-unitary_ii}
Chenyan Wu.
\newblock On $(\chi,b)$-factors of cuspidal automorphic representations of
  unitary groups, {II}.
\newblock preprint. arXiv:2104.11954v1. Submitted, 2021.

\bibitem[Yam14]{MR3211043}
Shunsuke Yamana.
\newblock L-functions and theta correspondence for classical groups.
\newblock {\em Invent. Math.}, 196(3):651--732, 2014.

\end{thebibliography}
\bibliographystyle{alpha}

\end{document}